\documentclass[11pt]{article}
\usepackage{amsmath,amssymb,amsthm,eucal,graphics}
\usepackage[utf8]{inputenc}
\usepackage[english]{babel}
\usepackage{amsfonts,amssymb,eucal}
\usepackage{amsbsy}
\usepackage{latexsym}

\usepackage{cite}
\usepackage{csquotes}
\usepackage{amsmath}
\usepackage{amsfonts}
\usepackage{amscd}
\usepackage{hyperref}
\usepackage{amssymb}
\usepackage{bbm}

\makeatletter
\gdef\firstpage{1}

\def\frsthdr{}

\def\firstpageone{0\thepage}
\def\firstpagetwo{00\thepage}
\def\firstpagethree{000\thepage}
\def\firstpagemark{\ifnum\firstpage <10  \firstpageone
\else\ifnum\firstpage<100 \firstpagetwo \else \ifnum\firstpage
<1000 \firstpagethree \else \firstpageone\fi\fi\fi}

\def\footline{\ifnum\thepage=\firstpage \footlineone
\gdef\firstpage{0}
\else\footlinetwo\fi}
\def\footlineone{\noindent \footnotesize \sf \hspace{\fill}  \hbox{}}
\def\footlinetwo{}

\def\titles{{Tensor completions of 2-nilpotent finitely generated torsion-free groups}}
\def\authors{{M. G. Amaglobeli, A. G. Myasnikov}}

\def\oddhedr{\ifnum\thepage=\firstpage \firsthdr \else \odhdr \fi}
\def\firsthdr{\hspace{\fill} \sl \frsthdr \hspace{\fill}\hbox{}}
\def\odhdr{\hspace{\fill}\sl\rightmark \titles \hspace{\fill}
\rm \thepage}

\def\evnhedr{\ifnum\thepage=\firstpage \firsthdr \else \evhdr \fi}
\def\evhdr{\noindent \rm \thepage\hspace*{\fill} \sl\leftmark
\authors \hspace*{\fill}\hbox{}}

\def\ps@newpstyle{\def\@oddhead{
\hspace{-0.65em} \vbox{\oddhedr\vskip 1mm \hrule width
\textwidth}
}
\def\@evenhead{
\hspace{-0.65em} \vbox{\evnhedr\vskip 1mm \hrule width
\textwidth}
}\textsc{}
\def\@oddfoot{\footline}
\def\@evenfoot{\@oddfoot}
}

\def\refer
{
\begin{small}
\begin{enumerate}
\leftmargin=0pt
\rightmargin=0pt
\itemsep=0pt
\parsep=0pt
}

\def\endref{\end{enumerate}\end{small} }

\usepackage{amsmath}
\usepackage{amssymb}
\usepackage{amsthm, amscd}
\usepackage[usenames]{color}

\usepackage{amsmath}

\usepackage{amssymb}
\usepackage{amsfonts}
\usepackage{amscd}
\usepackage{hyperref}

\usepackage[T2A]{fontenc}
\usepackage[all]{xy}
\usepackage[utf8]{inputenc}
\usepackage[english]{babel}

\def\Z{{\mathbb Z}}
\def\N{{\mathbb N}}
\def\Q{{\mathbb Q}}

\def\M{{\mathcal{M}}}
\def\N{{\mathcal{N}}}

\newtheorem{theorem}{Theorem}[section]
\newtheorem{lemma}{Lemma}[section]
\theoremstyle{definition}
\newtheorem{definition}{Definition}[section]

\newtheorem*{remark}{Remark}

\newtheorem{problem}{Problem}[section]

{}
{}

\usepackage{cite}
\usepackage{csquotes}
\usepackage{amsmath}
\usepackage{amsfonts}
\usepackage{amscd}
\usepackage{hyperref}
\usepackage{amssymb}
\usepackage{bbm}

\usepackage{hyperref}
\usepackage[usenames,dvipsnames,svgnames,table,rgb]{xcolor}
\hypersetup{				% Гиперссылки
    unicode=true,           % русские буквы в раздела PDF
    pdftitle={Заголовок},   % Заголовок
    pdfauthor={Автор},      % Автор
    pdfsubject={Тема},      % Тема
    pdfcreator={Создатель}, % Создатель
    pdfproducer={Производитель}, % Производитель
    pdfkeywords={keyword1} {key2} {key3}, % Ключевые слова
    colorlinks=true,       	% false: ссылки в рамках; true: цветные ссылки
    linkcolor=black,          % внутренние ссылки
    citecolor=black,        % на библиографию
    filecolor=magenta,      % на файлы
    urlcolor=cyan           % на URL
}

\usepackage{geometry}
\footnotesize

\begin{document}

\setcounter{page}{\firstpage}
\pagestyle{newpstyle}

\sloppy
\rm
%\begin{flushright}
%UDK: 512.544.33
%\end{flushright}

% ----------------------------

%\thispagestyle{plain}
\begin{center}
    {\Large
   \textbf{Tensor completions of 2-nilpotent finitely generated torsion-free groups} \newline

   M. G. Amaglobeli \footnote {Department of Mathematics, Iv. Javakhishvili Tbilisi State University, Georgia, e-mail: mikheil.amaglobeli@tsu.ge,}, A. G. Myasnikov \footnote {Stevens Institute of Technology, Hoboken, NJ 07030, USA, e-mail: amiasnikov@gmail.com}} \newline

%On the 85th anniversary of Yu. L. Ershov
\end{center}

\title {Tensor completions of 2-nilpotent finitely generated torsion-free groups% \footnote {This work was supported by the Shota Rustaveli National Science Foundation of Georgia under project FR-21-4713.}
}

\author {Mikheil Amaglobeli \footnote {Department of Mathematics, Iv. Javakhishvili, Georgia, email: mikheil.amaglobeli@tsu.ge,}, Alexey Myasnikov \footnote {Department of Mathematical Sciences, Stevens Institute of Technology, Castle Point-on-Hudson, Hoboken, NJ 07030, USA, email: amiasnikov@gmail.com},}

% \maketitle

\begin{abstract}
In this paper, we study tensor completions $G \otimes_{\N_{2,R}} R$ of finitely generated torsion-free nilpotent groups $G$ of class $2$ in the quasivariety $\N_{2,R}$ of $R$-exponential 2-nilpotent groups over a binomial integral domain $R$. We show that the classical Hall completion $G\otimes_{\mathcal{H}} R$ embeds as an abstract group (the embedding is not an $R$-homomorphism) into $G \otimes_{\N_{2,R}} R$, such that $G\otimes_{\N_{2,R}} R \simeq (G \otimes_{\mathcal{H}} R) \times D$, where $D$ is an $R$-module and the direct product is a product of abstract groups (not $R$-groups!). In particular, the canonical $R$-epimorphism $\mu: G \otimes_{\N_{2,R}} R \to G \otimes_{\mathcal{H}} R$ is a retract on $G \otimes_{\mathcal{H}} R$ with abelian kernel $D$. Moreover, in addition to the algebraic structure, we describe precisely how raising to an $R$-exponent works in the group $G \otimes_{\N_{2,R}} R$. To do this, we introduce a new type of commutators, the so-called c-commutators, which are interesting in their own right. These results answer an old question of Remeslennikov about the algebraic structure of free 2-nilpotent R-groups in the quasivariety $\N_{2,R}$. Indeed, it was shown in \cite{AMN} that if $G$~ is a free 2-nilpotent group with basis $X$ (in the variety of abstract 2-nilpotent groups), then $G \otimes_{\N_{2,R}} R$~ is a free 2-nilpotent R-group in $\N_{2,R}$ with basis $X$. Note that in this case $G \otimes_{\mathcal{H}} R$ is a free 2-nilpotent Hall $R$-group with basis $X$. As an illustration, for a free 2-nilpotent group $G$ of rank 2, we describe the group $G \otimes_{\N_{2,R}} R$, the action of $R$ on $G \otimes_{\mathcal{H}} R$, and the module $D$ in the case where $R$ is either the polynomial ring $\mathbb{Q}[t]$ or the field of rational functions $\mathbb{Q}(t)$ with coefficients in the field of rational numbers $\mathbb{Q}$.
\end{abstract}

% Keywords: nilpotent groups, $R$-completions, Hall completions, retracts.

\section{Introduction}
In this paper, we study tensor completions $G \otimes_{\N_{2,R}} R$ of finitely generated torsion-free 2-nilpotent groups $G$ in the class $\N_{2,R}$ of all 2-nilpotent $R$-groups over a binomial domain $R$. We show that the classical Hall completion $G\otimes_{\mathcal{H}} R$ embeds as an abstract group (the embedding is not an $R$-homomorphism) in $G \otimes_{\N_{2,R}} R$, such that $G\otimes_{\N_{2,R}} R \simeq (G \otimes_{\mathcal{H}} R) \times D$, where $D$ is an $R$-module and the direct product is a product of abstract groups. In particular, the canonical $R$-epimorphism $\mu: G \otimes_{\N_{2,R}} R \to G \otimes_{\mathcal{H}} R$ is a retract onto $G \otimes_{\mathcal{H}} R$ with abelian kernel $D$. Moreover, in addition to the algebraic structure, we precisely describe how raising to an $R$-exponent works in the group $G \otimes_{\N_{2,R}} R$. To this end, we introduce a new type of commutators, the so-called c-commutators, which are interesting in their own right. This answers an old question by Remeslennikov about the algebraic structure of free 2-nilpotent $R$-groups in the quasivariety $N_{2,R}$ (they are precisely the tensor $R$-completions of free 2-nilpotent groups). As an illustration, for a free 2-nilpotent group $G$ of rank 2, we describe the group $G \otimes_{\N_{2,R}} R$, the action of $R$ on $G \otimes_{\mathcal{H}} R$, and the module $D$ in the case where $R$ is either the polynomial ring $\mathbb{Q}[t]$ or the field of rational functions $\mathbb{Q}(t)$ with coefficients in the field of rational numbers $\mathbb{Q}$. This result was previously obtained by different methods by Amaglobeli and Remeslennikov in \cite{3}.

Recall that a group $G$ is called a \emph{divisible group with unique roots} if, for any $g \in G$ and $0 \neq n \in \mathbb{N}$, the equation $x^n = g$ has a unique solution in $G$. In this case, there exists a $\Q$-exponentiation on $G$, that is, a mapping $G\times \Q \to G$ defined by $(g,\frac{m}{n}) \to (g^{\frac{1}{n}})^m$, which satisfies the following properties for any $x,y \in G$ and $\alpha,\beta \in \Q$:
\begin{enumerate}
\item [1)] $x^1=x$, $x^0=e$, $e^\alpha=e$;
\item [2)] $x^\alpha x^\beta=x^{\alpha+\beta}$, $(x^\alpha)^\beta=x^{\alpha \beta}$;
\item [3)] $(y^{-1}xy)^\alpha=y^{-1}x^\alpha y$.
\end{enumerate}
We will call such groups $\Q$-groups. Mal'tsev proved that divisible torsion-free nilpotent groups are $\Q$-groups (see \cite{Malcev1,Malcev2}). A thorough study of arbitrary $\Q$-groups was carried out by Baumslag in \cite{Baumslag}. Around the same time, while studying equations in free groups, Lyndon \cite{1} introduced the concept of an $R$-group (and coined this name) for an arbitrary associative ring $R$ with unity. By definition, these are groups equipped with a map $G \times R \to G$, denoted $(g,\alpha) \to g^\alpha $, that exactly satisfy Axioms 1) --- 3) above.

Unfortunately, there are abelian Lyndon $R$-groups that are not $R$-modules. Later, Myasnikov and Remeslennikov in their paper~\cite{2} added an additional scheme of axioms (quasi-identities) to the Lyndon axioms for any $x,y \in G$ and any $\alpha \in R$ (here $[g,h] = g^{-1}h^{-1}gh$):

\begin{enumerate} \label{eq:4}
\item [4)] \text{(\emph{MR}-axiom)} \ \ \ $[g,h]=e\longrightarrow (gh)^{\alpha}=g^{\alpha}h^{\alpha}$.
\end{enumerate}

Note that all $\mathbb{Q}$-groups satisfying Axioms 1)-3) also satisfy Axiom 4), as do many other natural $R$-groups according to Lyndon. Most likely, it was precisely because of this circumstance that Axiom 4) was not added to the original definition of an $R$-group. We denote by $\mathcal{L}_{R}$ the class of all Lyndon $R$-groups, i.e., those satisfying Axioms 1)-3), and by $\mathcal{M}_{R}$ the subclass of those groups from $\mathcal{L}_{R}$ that satisfy \emph{MR}-axiom 4). In what follows, unless otherwise stated, by $R$-groups we mean groups from the class $\mathcal{M}_{R}$.

The tensor extension of a scalar ring plays an important role in module theory.
In \cite{2}, Myasnikov and Remeslennikov defined a precise analog of this construction, the so-called tensor completion, for an arbitrary group, and in \cite{6} they described a concrete method for constructing tensor completions for CSA groups, in particular, for free and torsion-free hyperbolic groups (see also \cite{BMR1}). However, the concept of tensor completion in an arbitrary variety of groups requires a special approach, since the general construction of tensor completions is not necessarily compatible with a given variety. Amaglobeli, Myasnikov, and Nadiradze introduced the general concept of a variety of $R$-groups and tensor completions in such varieties in \cite{AMN, 10}. It turned out that different, albeit equivalent, definitions of nilpotency in the class of discrete groups can lead to formally different types of nilpotent $R$-groups. The question of whether these classes are truly distinct remains open. However, it is known \cite{AMN} that in the class of 2-nilpotent groups, different definitions of nilpotency lead to the same previously mentioned quasivariety of 2-nilpotent $R$-groups $\N_{2,R}$. It is with respect to this class that we describe the tensor completions $G \otimes_{\N_{2,R}} R$ of finitely generated 2-nilpotent torsion-free groups $G$.

We will now briefly describe the structure of the paper. In Section 2, we discuss some important concepts and recap some relevant facts.

In Section 3, we introduce the basic technical notion of c-commutators $c(g,h)_\alpha$ in 2-nilpotent $R$-groups, where $g$ and $h$ are group elements and $\alpha \in R$, which measure the deviation of raising $gh$ to the $\alpha$-exponent in groups from the class $\N_{2,R}$ from the Hall $\alpha$-exponentiation. We show that many properties of $R$-exponentiation in groups can be expressed in the language of c-commutators. This results in a calculus of c-commutators similar to the classical commutator calculus for Lie algebras, but on group exponents. Such commutators and their interesting connections with ordinary commutators provide a powerful tool for studying exponentiation in groups. We have studied c-commutators only in the case of 2-nilpotent groups, but they can be defined in any nilpotent group, as well as in arbitrary groups, by considering factors by the lower central series. These are new concepts, and their successful application requires solving several fundamental questions: developing the calculus of c-commutators for arbitrary nilpotency class, in particular, obtaining analogues of Petresco words and developing a collecting process similar to the classical collecting process for ordinary commutators, describing tensor completions in the classes $\N_{c,R}$ for $c \geq 3$, etc. It is known from the classical Mal'tsev correspondence that nilpotent Lie $R$-algebras correspond to Hall nilpotent $R$-groups over a field $R$ of characteristic zero. It would be interesting to understand which "generalized Lie algebras" correspond to groups from the class $\N_{c,R}$. We discuss some of these issues in the final section 7.

In Section 4, we prove the main theorem on the group structure of $G \otimes_{\N_{2,R}} R$, where $G$ is a finitely generated 2-nilpotent torsion-free group and $R$ is an arbitrary binomial integral domain. Namely, in the notation above, we show that $G\otimes_{\N_{2,R}} R \simeq (G \otimes_{\mathcal{H}} R) \times D$, where $D$ is the $R$-module generated by all c-commutators of the form $c(g,h)_\lambda$ in $G \otimes_{\N_{2,R}} R$.

In Section 5, we describe the $R$-structure of the group $G \otimes_{\N_{2,R}} R$; that is, we precisely describe the function $(g,\lambda) \to g^\lambda$ of $R$-exponentiation in the group $G \otimes_{\N_{2,R}} R$, give an exact representation of the $R$-module $D$ in terms of generators and relations, and prove that $D$ is a free $R$-module. However, finding a basis for this module is generally a rather difficult problem that is beyond the scope of this article. Indeed, the structure of such a basis depends essentially on both $G$ and $R$.
In Section 6, we give two examples of the structure of the module $D$ when the group $G$ is the group $UT_3(\mathbb Z)$; that is, the free 2-nilpotent group of rank 2, and the ring $R$ is either the polynomial ring $\mathbb{Q}[t]$ or the field of rational functions $\mathbb{Q}(t)$.

\section{Preliminaries}

\subsection{Hall \texorpdfstring{$R$}{}-Groups}
\label{se:2.1}

In \cite{12}, Hall studied nilpotent $R$-groups satisfying the following additional axiom for all $x_1, \ldots,x_n \in G$ and $\alpha \in R$:
\begin{enumerate}
\item [5)] $x_1^\alpha x_2^\alpha\cdots x_n^\alpha=(x_1x_2\cdots x_n)^\alpha\tau_2(\bar{x})^{\binom{\alpha}{2}}\cdots \tau_c(\bar{x})^{\binom{\alpha}{c}}$,

where $\tau_i$~ are Petresco group words over the alphabet $x_1, \ldots,x_n$, $c$~--- the nilpotency class of $G$, and ${\binom{\alpha}{i}}$~--- the binomial coefficients.
\end{enumerate}
In this case, the ring $R$ must be a binomial domain, i.e., an integral domain, where for each $\alpha \in R$, the binomial coefficient
${\binom{\alpha}{i}} = \frac{\alpha(\alpha-1)\ldots (\alpha-i+1)}{i!}$, considered as an element of the field of fractions of $R$, belongs to $R$.

It is known that the word $\tau_i(\bar x)$, considered an element of the free group $F$ with basis $x_1, \ldots,x_n$, belongs to the $i$-th term of the lower central series of $F$. So, if $[x,y] =e$ in $G$, then $\tau_i(x,y) = e$ for all $i > 1$. This implies that   Axiom 5) is a much stronger and more precise version of Axiom 4) in the case of nilpotent groups. These groups, now known as Hall $R$-groups, have become the main object of study in the theory of nilpotent $R$-groups. In a sense, they are true nilpotent modules over a ring $R$. The class of all Hall $R$-groups that are nilpotent of step $\leq c$ is denoted by $\mathcal H_{c,R}$. For the theory of Hall nilpotent groups, we refer to the books by Hall \cite{12} and Warfield \cite{W}. Here we would like to note one important point. All nilpotent $\Q$-groups satisfying Axioms 1)-3) also satisfy Axioms 4) and 5) (see section \ref{se:k-groups}). Therefore, nilpotent Mal'tsev $\Q$-groups are exactly the same as nilpotent Hall $\Q$-groups. However, as the main theorem below shows, this is not true for an arbitrary field $k$ of characteristic zero (see also \cite{3, 9}).

\subsection{\texorpdfstring{$R$}{}-Groups}

In this section, following the works of \cite{1,2,AMN}, we introduce the necessary terminology and discuss some properties of $R$-groups that are useful for further understanding.

Let $R$ be an associative commutative ring with identity and $G \in \M_R$. The standard group language $L_{gr}$ can be extended by symbols for the unary operations $f_\alpha$ for each $\alpha \in R$, yielding the language of $R$-groups: $L_{gr}^R = L_{gr} \cup \{f_\alpha \mid \alpha \in R\}$. The operation $f_\alpha$ is interpreted in an $R$-group $G$ by raising to the $\alpha$-exponent: $f_\alpha(g) = g^\alpha$ for $g \in G$. As usual in universal algebra, the notions of $R$-homomorphisms, $R$-subgroups, and $R$-generations can be defined in the category $\M_R$ by considering them as groups in the language $L_{gr}^R$. Thus, a homomorphism $\phi:G\rightarrow H$ of two $R$-groups $G$ and $H$ is an $R$-homomorphism if $\phi(g^\alpha)=(\phi(g))^\alpha$ for all $g\in G$ and $\alpha\in R$; A subgroup $N \leq G$ is called an \emph{$R$-subgroup} of $G$ if it is invariant under raising to the $R$-exponent, and a subset $X \subseteq G$ \emph{$R$-generates} an $R$-subgroup $\langle X\rangle_R$, which is the smallest $R$-subgroup of $G$ containing $X$.
Note that $G$ is generated (as an $R$-group) by a set $A \subseteq G$ if every element of $G$ can be represented as $w(a_1, \ldots,a_n)$ for some group $R$-word $w(z_1, \ldots,z_n)$ and $a_1, \ldots, a_n \in A$. Here, if $z_1, \ldots, z_n$ are variables, then any expression $w(z_1, \ldots,z_n)$ obtained from these variables by multiplication, inversion, and exponentiation by elements of $R$ is called a \emph{group $R$-word in the variables $z_1, \ldots,z_n$} (group $R$-words are precisely terms in the language $L_{gr}^R$).

Axioms 1)-3) and 5) are identities in the language $L_{gr}^R$, while Axiom 4) is a quasi-identity. Therefore, the class of all Lyndon $R$-groups, $\mathcal{L}_R$, like the class $\mathcal{H}_{c,R}$ of all Hall groups of nilpotency class $\leq c$, is a variety in the language of $L_{gr}^R$, while the classes $\M_R$ and $\N_{2,R}$ are quasivarieties. It follows from general theorems of universal algebra that these classes contain free $R$-groups and tensor $R$-completions with respect to these classes; the standard theory of $R$-ideals, i.e., normal subgroups whose quotient groups again belong to the same class; we can speak of groups defined by generators and relations in these classes, etc.

\subsection{Ideals and \texorpdfstring{$R$}{}-commutators}
\label{se:2.2}

By definition, $R$-ideals in a given quasivariety of $R$-groups are normal $R$-subgroups whose quotient groups again lie in the same quasivariety. Since $\mathcal{L}_{R}$ and $\mathcal{H}_{c,R}$ are varieties of $R$-groups, $R$-ideals in them are simply normal $R$-subgroups; while $\M_R$ and $\N_{c,R}$ are quasivarieties, so to define $R$-ideals here, additional conditions are required that ensure the fulfillment of Axiom 4) in the quotient group. For this purpose, the notion of $\alpha$-commutator was introduced in \cite{2}.
Namely,
let $G\in\mathcal{L}_R$, $g,h\in G$, $\alpha\in R$. Then the element
\begin{equation} \label{def:1.4}
      (g,h)_{\alpha}=h^{-\alpha}g^{-\alpha}(gh)^{\alpha}   
\end{equation}
is called the \emph{$\alpha$-commutator} of $g$ and $h$. We call such commutators \emph{$R$-commutators} if we do not want to specify the element $\alpha \in R$. Clearly, $(gh)^{\alpha} =g^{\alpha}h^{\alpha}(g,h)_{\alpha}$, and hence the $\alpha$-commutators $(g,h)_\alpha$ measure the "deviation
of the exponentiation in $G$ from the commutative one". Therefore, in abelian $R$-groups $G \ $ $(g,h)_\alpha = e$ for all $g,h \in G$ and $\alpha \in R$, and in Hall groups $G \in \mathcal{H}_{c,R}$, by Axiom 5),
$$
(g,h)_\alpha = \tau_c(g,h)^{-\binom{\alpha}{c}} \ldots \tau_2(g,h)^{-\binom{\alpha}{2}}.
$$
In particular, for $G \in \mathcal{H}_{2,R}$ we obtain
\begin{equation} \label{eq:Hall-commutator}
    (g,h)_\alpha = \tau_2(g,h)^{-\binom{\alpha}{2}}.
\end{equation}

Note that $(g,h)_{-1} =[h^{-1},g^{-1}]$, so every commutator is also an $\alpha$-commutator. By definition, if $G\in\mathcal{L}_R$, then
$$
 G\in\mathcal{M}_R \Longleftrightarrow \big([g,h]=e\Longrightarrow (g,h)_{\alpha}=e\big).
$$

The last equivalence leads to the definition of the $\mathcal{M}_R$-ideal.

\begin{definition}(\cite{2})\label{def:1.5}
Let $G\in\mathcal{L}_R$. A normal $R$-subgroup $H\trianglelefteq G$ is called an \emph{$R$-ideal} if for any $g,h\in G$, if  $[g,h]\in H$, then $(g,h)_{\alpha}\in H$ for any $\alpha\in R$.
\end{definition}

\subsection{Nilpotent \texorpdfstring{$R$}{}-Groups}\label{se:k-groups}

For a positive integer $c>1$, we denote by $\mathcal{N}_{c,R}$ the quasivariety of all nilpotent $R$-groups in which the $c$-nilpotency identity
$$  \forall\,x_1,\dots,x_{c+1}\quad [x_1,\dots,x_{c+1}]=e        $$
holds, where $[x_1,\dots,x_{c+1}]$ is the left-normed commutator of weight $c+1$.
It is clear that $\mathcal{H}_{c,R} \subseteq \N_{c,R}$.

The fundamental difference between nilpotent $R$-groups in $N_{c,R}$ and Hall $R$-groups is as follows.
Axiom 5 (the Hall-Petresco formula) allows one to rewrite any group $R$-word $w(z_1, \ldots,z_n)$ into an equivalent word $u(z_1, \ldots,z_n)$ of the form $y_1^{\alpha_1} \ldots y_m^{\alpha_m}$, where each $y_i$ is one of the variables $z_1, \ldots, z_n$ and $\alpha_i \in R$. Whereas in arbitrary nilpotent $R$-groups, this is not the case. Here, two group $R$-words $w(z_1, \ldots,z_n)$ and $u(z_1, \ldots,z_n)$ are equivalent over $G$ if, for any $g_1, \ldots,g_n \in G$, we have $w(g_1, \ldots,g_n) = u(g_1, \ldots,g_n)$.
More precisely, the following statement, proved in the paper \cite{ABM}, holds.

\begin{lemma} \label{le:rat-exp}
	 Let $c\geq 1$ be a fixed positive integer. Then every group $R$-word $w(z_1, \ldots,z_n)$ is equivalent to an $R$-word of the form $y_1^{\alpha_1} \ldots y_m^{\alpha_m}$, where $y_i$ is one of the variables $z_1, \ldots, z_n$ and $\alpha_i \in R$, over any Hall nilpotent $R$-group $G$ of nilpotency class $\leqslant c$.
		
\end{lemma}

\subsection{Tensor Completions of Nilpotent \texorpdfstring{$R$}{}-Groups}
\label{se:tensor-completions}

The operation of tensor completion plays a crucial role in the study of $R$-groups. In this section, we discuss tensor completions in the classes $\N_{c,R}$ and $\mathcal{H}_{c,R}$.

Let $G$ be a torsion-free nilpotent group of nilpotency class $c\geq 1$, and let $R$ be an associative commutative ring with 1 of characteristic 0 (in the case of Hall groups, we also assume that $R$ is a binomial domain). In particular, the embedding $\Z \leq R$ is fixed.

Recall that an $R$-group $G \otimes_{\N_{c,R} }R$ is called a \emph{tensor $R$-completion of $G$ in the class $\N_{c,R}$} if there exists a homomorphism $\lambda:G\to G \otimes_{\N_{c,R} }R$ such that the following conditions are satisfied:

\begin{enumerate}
\item[{\rm (1)}] $\lambda(G)$ \ $R$-generates $G \otimes_{\N_{c,R}}R$;

\item[{\rm (2)}] for any $R$-group $H \in \N_{c,R}$ and an arbitrary homomorphism $\varphi:G\to H$, there exists an $R$-homomorphism $\psi: G \otimes_{\N_{c,R}} R\to H$ that makes the following diagram commutative:
$$ \begin{matrix}
\xymatrix{
G\ar[r]^{\lambda\quad} \ar[d]_{\varphi} & G \otimes_{\N_{c,R}} R \ar@{-->}[dl]^{\psi} \\
H & }\end{matrix} \;\;\; (\varphi=\psi\circ \lambda). $$
\end{enumerate}

It was proved in \cite{AMN} that for any such group $G$ and any such ring $R$, the tensor completion $G \otimes_{\N_{c,R}}R$ exists and is unique up to $R$-isomorphism. Moreover, the homomorphism $\lambda$ gives a natural embedding of $G$ into $G \otimes_{\N_{c,R}}R$. If $G$ is an abelian $R$-group (i.e., $c= 1$), then $G \otimes_{\N_{c,R}}R\cong G\otimes {R}$ is the tensor product of the $\Z$-module $G$ and the ring $R$.

Tensor $R$-completions $G \otimes_{\mathcal{H}_{c,R}} R$ of a group $G$ in the class $\mathcal{H}_{c,R}$ are also well known and are described by Hall in \cite{12}. Namely, let $G$ be finitely generated (the general case is treated similarly; one can also use direct limits of finitely generated groups \cite{2}). Then $G$ has a tuple of elements $\bar w = (w_1, \ldots, w_n)$ (the Mal'tsev basis of $G$) such that every element $g$ of $G$ is uniquely representable as
\begin{equation} \label{eq:Malcev-basis}
g= w_1^{t_1} \ldots w_n^{t_n}
\end{equation}
for some $t_i \in \Z$. The tuple $t(g) = (t_1, \ldots,t_n)$ is called the tuple of Mal'tsev coordinates $g$. In this case, the multiplication of elements of the form (\ref{eq:Malcev-basis}) is performed using some polynomials $f_1(\bar x, \bar y), \ldots, f_n(\bar x,\bar y)$, where $\bar x = (x_1, \ldots,x_n), \bar y = (y_1, \ldots,y_n)$ such that for any $g, h \in G$
$$
t(gh) = (f_1(t(g),t(h)), \ldots, f_n(t(g),t(h))).
$$
Similarly, through polynomials, raising to the exponent $\lambda \in \Z$ is introduced (see \cite{12}).
To construct $G \otimes_{\mathcal{H}_{c}} R$, we take the same Mal'tsev basis
$\bar w = (w_1, \ldots, w_n)$ of $G$, consider all formal products
$$
w_1^{\alpha_1} \ldots w_n^{\alpha_n},
$$
and introduce multiplication and exponentiation $\alpha \in R$ on them using the same polynomials as for $G$. Since $R$ is binomial, all values of these polynomials are defined in $R$ (see details in \cite{W}).

%%%%%%%%%%%%%%%%%%%%%%%%%%%
%%%%%%%%%%%%%%%%%%%%%%%%%%%%%%%

\section{Commutator Identities in 2-Nilpotent \texorpdfstring{$R$}{R}-Groups}
\label{sec:commutators}

In this section, we introduce the concept of a \emph{$c$-commutator} and establish some useful relations in 2-nilpotent $R$-groups in the language of $\alpha$-commutators and c-commutators. In particular, we rewrite Axioms 1)-4) of an exponential $R$-group in the language of such commutators. To precisely track which axioms follow from which relations, we fix a 2-nilpotent group $G$, in which we define a function $G\times R \to G$, denoted $(g,\alpha) \to g^\alpha$. We do not assume any other properties of this function in advance. The ring $R$ is still assumed to be binomial.

\begin{definition}
For $g,h \in G$ and $\alpha \in R$, we define the \emph{$c$-commutator}
\begin{equation} \label{eq:def-c-comm}
c(g,h)_\alpha = [g,h]^{\binom{\alpha}{2}}(g,h)_\alpha.
\end{equation}

\end{definition}
Then
\begin{equation} \label{eq:exp-c-comm}
(gh)^\alpha = g^\alpha h^\alpha [g,h]^{-\binom{\alpha}{2}} c(g,h)_\alpha
\end{equation}
and hence, the element $c(g,h)_\alpha$ \emph{measures how much exponentiation  $(gh)^\alpha$ in $G$ differs from the Hall exponentiation.}

Before addressing Axioms 1)---4), we note several general facts about 2-nilpotent $R$-groups.

\medskip
\noindent
{\bf F0)} [{\it Weak form of Axiom 4)}] For any 2-nilpotent $R$-group $G$, if $z \in Z(G)$, then for any $g \in G, \alpha \in R$
\begin{equation} \label{eq:F0}
(gz)^\alpha = g^\alpha z^\alpha.
\end{equation}

\medskip
Below, we identify $n \in \mathbb{Z}$ with $1\cdot n \in R$ and $r = \frac{m}{n} \in \mathbb{Q}$ with $mn^{-1}$ in $R$, provided that $n$ is invertible in $R$. The set of all such $mn^{-1} \in R$ is denoted by $\mathbb{Q} \cap R$.

\medskip
\noindent 
{\bf F1)} [{\it Rational exponents are always Hall exponents}] 
For any 2-nilpotent $R$-group $G$ and for any $g,h \in G, r \in \mathbb{Q} \cap R$
\begin{equation} \label{eq:rationalb}
    c(g,h)_{r}= e. 
\end{equation}
Indeed, as we mentioned in Section \ref{se:2.1}, exponentiation by rational numbers in any $R$-group is Hall (see \cite{ABM}).

\medskip
\noindent 
{\bf F2).} [{\it Axiom 3)}] Suppose that the function $(g,\alpha) \to g^\alpha$ satisfies F0) in the group $G$. Then for any $g, h \in G$, $\alpha, \beta \in R$
\begin{equation} \label{eq:14}
(h^{-1}gh)^\alpha = h^{-1}g^\alpha h \Longleftrightarrow [g^\alpha,h] = [g,h]^{\alpha}.
\end{equation}

Indeed, if $(h^{-1}gh)^\alpha = h^{-1}g^\alpha h$, then using $[g,h] \in Z(G)$ (from the 2-nilpotency of $G$), we obtain that, on the one hand,  

$$
(h^{-1}gh)^\alpha = (gg^{-1}h^{-1}gh)^\alpha = (g[g,h])^\alpha = g^\alpha[g,h]^\alpha.
$$
On the other hand,
$$
h^{-1}g^\alpha h = g^\alpha g^{-\alpha}h^{-1}g^\alpha h = g^\alpha[g^\alpha,h].
$$
From which, $[g,h]^\alpha = [g^\alpha,h]$. It is clear that the chain of equalities above is invertible, which means that the converse statement is also true.

\medskip
As a consequence of F2) we obtain

\medskip
\noindent 
{\bf F3)} Suppose that the function $(g,\alpha) \to g^\alpha$ satisfies F0) and Axiom 3) in the group $G$. Then for any $g,h \in G, \alpha \in R$
\begin{equation} \label{eq:rationalb2}
    [g^\alpha,h] = [g,h]^{\alpha}. 
\end{equation}
It follows that for any $g,h \in G, \alpha, \beta \in R$
$$
  [g^\alpha,h^\beta] = [g,h]^{\alpha \beta}.
$$

\medskip
\noindent 
Note that from F3) it follows that the commutator subgroup $[G,G]$ of any 2-nilpotent $R$-group is an $R$-subgroup of $G$.

\medskip
\noindent 
{\bf F4)} {\it  Suppose that the function $(g,\alpha) \to g^\alpha$ satisfies F3) in the group $G$. Then for any 2-nilpotent $R$-group $G$ and for any $g, h \in G$, $\alpha \in R$
\begin{equation} \label{eq:fact3}
   (g,h)_\alpha \in Z(G) \ \ \text{и} \ \  c(g,h)_\alpha \in Z(G).
\end{equation}
}
Indeed, according to the formula (\ref{eq:def-c-comm}) it is sufficient to verify that the $\alpha$-commutator $(g,h)_{\alpha}$ lies in the center of $G$:
\begin{multline*}
\big[(g,h)_{\alpha},z\big]=\big[h^{-\alpha}g^{-\alpha}(gh)^{\alpha},z\big]=[h^{-\alpha},z]\,[g^{-\alpha},z]\,\big[(gh)^{\alpha},z\big]= \\
    =[h,z]^{-\alpha}[g,z]^{-\alpha}[gh,z]^{\alpha}=[h,z]^{-\alpha}[g,z]^{-\alpha}[g,z]^{\alpha}[h,z]^{\alpha}=e. \qedhere
\end{multline*}

From F4) as a consequence we obtain,

\medskip
\noindent 
{\bf F5)} {\it The center $Z(G)$ is an $R$-ideal of $G$.} 

\medskip
\noindent 
{\bf F6)} {\it Suppose that the function $(g,\alpha) \to g^\alpha$ satisfies F0) and F3) in the group $G$. Then, if $z \in Z(G)$, then for any $g,f \in G$, $\alpha \in R$
\begin{equation} \label{eq:F5a}
    c(g,fz)_\alpha = c(g,f)_\alpha.
    \end{equation}
In particular,
\begin{equation} \label{eq:F5b}
    c(g,z)_\alpha = e,
    \end{equation}
    and also
  \begin{equation} \label{eq:left-comm}
    c(fh,g)_\alpha = c(hf,g)_\alpha
\end{equation}  
}

Indeed, 
$$
(g,fz)_\alpha = (fz)^{-\alpha}g^{-\alpha} (gfz)^\alpha = f^{-\alpha}z^{-\alpha}g^{-\alpha}(gf)^\alpha z^\alpha = f^{-\alpha}g^{-\alpha}(gf)^\alpha = (g,f)_\alpha.
$$
Therefore, according to formula (\ref{eq:def-c-comm}), $c(g,fz)_\alpha = c(g,f)_\alpha$. Equality (\ref{eq:left-comm}) follows from (\ref{eq:F5a}), since $fh = hf[f,h]$ and $[f,h] \in Z(G)$.

\medskip
\noindent
{\bf F7)} [{\it Commutativity of the ring $R$}] {\it Suppose that the function $(g,\alpha) \to g^\alpha$ satisfies F0), F3), and Axiom 2.2 in the group $G$. Then the commutativity of the ring $R$ implies that the following identity for $c$-commutators holds in $G$: for any $f,h \in G$ and $\alpha,\beta \in R$
\begin{equation} \label{eq:6}
c(f,h)_{\alpha}^{\beta}c(f^{\alpha},h^{\alpha})_{\beta}
=c(f,h)_{\beta}^{\alpha}c(f^{\beta},h^{\beta})_{\alpha}.
\end{equation}
}

Indeed, by the commutativity of the ring $R$ and the satisfiability of Axiom 2.2, $((fh)^\alpha)^\beta = ((fh)^\beta)^\alpha$. Therefore, using F0), F3), F4) and F6) (the last two follow from F0) and F3)), we obtain
\begin{align*}
((fh)^{\alpha})^{\beta} & =\big(f^{\alpha}h^{\alpha}(f,h)_{\alpha}\big)^{\beta}=
f^{\alpha\beta}(h^{\alpha}(f,h)_{\alpha})^{\beta}(f^{\alpha},h^{\alpha}(f,h)_{\alpha})_{\beta} \\
& =f^{\alpha\beta}h^{\alpha\beta}(f,h)_{\alpha}^{\beta}(f^{\alpha},h^{\alpha})_{\beta}. \\ 
((fh)^{\beta})^{\alpha} & =\big(f^{\beta}h^{\beta}(f,h)_{\beta}\big)^{\alpha}=f^{\beta\alpha}(h^{\beta}(f,h)_{\beta})^{\alpha}(f^{\beta},h^{\beta}(f,h)_{\beta}){\alpha} \\ 
& =f^{\beta\alpha}h^{\beta\alpha}(f,h)_{\beta}^{\alpha}(f^{\beta},h^{\beta})_{\alpha}.
\end{align*}
Canceling both sides by $f^{\alpha\beta}h^{\alpha\beta}$, we obtain
$$
(f,h)_{\alpha}^{\beta}(f^{\alpha},h^{\alpha})_{\beta}
=(f,h)_{\beta}^{\alpha}(f^{\beta},h^{\beta})_{\alpha}.
$$
Now rewriting this equality in terms of $c$-commutators using formula (\ref{eq:def-c-comm}), we obtain (\ref{eq:6}).

To prove the converse, we additionally need Axiom 2.1 and the property that the group $G$ has no $R$-torsion. Then, given F2) and F3), the equalities above are invertible, which means $((fh)^\alpha)^\beta = ((fh)^\beta)^\alpha$. By Axiom 2.2, we obtain $f^{\alpha \beta} = f^{\beta \alpha }$ for any $f \in G, \alpha, \beta \in R$. By Axiom 2.1, $f^{\alpha\beta -\beta\alpha} = e$, which implies $\alpha\beta -\beta\alpha = 0$, since there is no $R$-torsion in $G$.

Note that (\ref{eq:6}) can also be deduced from F9) below.

We will now rewrite the axioms of exponential $R$-groups in 2-nilpotent $R$-groups in the language of $R$-commutators and $c$-commutators.

\medskip
\noindent
{\bf F8)} [{\it Axiom 2.1: $g^{\alpha+\beta} = g^{\alpha}g^{\beta}$}]
{\it Suppose that the function $(g,\alpha) \to g^\alpha$ satisfies F3) in the group $G$. Then for any $g, h \in G$, $\alpha \in R$ the following equivalence holds:
$$
g^{\alpha+\beta} = g^{\alpha}g^{\beta} \Longleftrightarrow c(g,h)_{\alpha +\beta} = c(g,h)_{\alpha}c(g,h)_{\beta}.
$$
}

Indeed, on the one hand
$$
(gh)^{\alpha+\beta}=g^{\alpha+\beta}h^{\alpha+\beta} (g,h)_{\alpha+\beta},
$$
and on the other hand
$$
(gh)^{\alpha+\beta}=(gh)^\alpha (gh)^{\beta}=g^\alpha h^\alpha (g,h)_\alpha g^\beta h^\beta (g,h)_\beta=g^{\alpha+\beta} h^{\alpha+\beta} [h^\alpha,g^\beta] (g,h)_\alpha (g,h)_\beta.
$$
Therefore, using F3), we obtain for any $g,h \in G$, $\alpha, \beta \in R$
\begin{equation}\label{eq:(1)}
(g,h)_{\alpha + \beta}=(g,h)_\alpha (g,h)_\beta [h,g]^{\alpha\beta}. \end{equation}
Rewriting this equality in terms of $c$-commutators, we obtain
\begin{equation} \label{eq:axiom2}
c(g,h)_{\alpha +\beta} = c(g,h)_{\alpha}c(g,h)_{\beta}.
\end{equation}
The converse is also true. If we reverse the steps of the previous proof, we see that Axiom 2.1 follows from equality (\ref{eq:axiom2}) and F3).

\medskip
\noindent
{\bf F9)} [{\it Axiom 2.2: $g^{\alpha \beta} = (g^{\alpha})^{\beta}$}]
{\it Suppose that the function $(g,\alpha) \to g^\alpha$ satisfies F0) and F3) in the group $G$. Then for any $g, h \in G$, $\alpha \in R$ the following equivalence holds:
$$
g^{\alpha \beta} = (g^{\alpha})^{\beta} \Longleftrightarrow c(g,h)_{\alpha\beta} = c(g^{\alpha},h^{\alpha})_\beta c(g,h)_{\alpha}^{\beta}
$$
}

On the one hand,
$$
(gh)^{\alpha\beta}=g^{\alpha\beta} h^{\alpha\beta}(g,h)_{\alpha\beta}.
$$
On the other hand, from F3) it follows that $(g,h)_\alpha \in Z(G)$, then using F0) we obtain for any $g \in G$ and $z \in Z(G)$

$$
(gh)^{\alpha\beta}=((gh)^\alpha)^\beta=(g^\alpha h^\alpha (g,h)_\alpha)^\beta=g^{\alpha\beta} h^{\alpha \beta} (g^\alpha,h^\alpha)_\beta (g,h)_\alpha^\beta.
$$
Therefore
\begin{equation}\label{eq:(2)}
(g,h)_{\alpha \beta}=(g^\alpha,h^\alpha)_\beta (g,h)_\alpha^\beta.
\end{equation}
It follows that
\begin{equation} \label{eq:axiom2b} 
c(g,h)_{\alpha\beta} = c(g^{\alpha},h^{\alpha})_\beta c(g,h)_{\alpha}^{\beta}
\end{equation}
The converse is also true: equality (\ref{eq:axiom2b}), given F0) and F3), implies Axiom 2.2).

\medskip
\noindent
{\bf F10)} [{\it Axiom 4) $[g,h] = e \Rightarrow (gh)^\alpha = g^\alpha h^\alpha$}] In a group $G$, Axiom 4) is equivalent to the following quasi-identity. For any $g,h \in G, \alpha \in R$
\begin{equation} \label{eq:axiom4b}
[g,h] = e \Rightarrow c(g,h)_\alpha = e.
\end{equation}

Indeed, in the language of $\alpha$-commutators, Axiom 4) is equivalent to the quasi-identity
$$
[g,h] = e \Rightarrow (g,h)_\alpha = e,
$$
which, in turn, is equivalent to (\ref{eq:axiom4b}).

\medskip
\noindent
{\bf F11)} [{\it Axiom 3) $(h^{-1}gh)^\alpha = h^{-1}g^\alpha h$ in terms of ordinary commutators}] {\it Suppose that the function $(g,\alpha) \to g^\alpha$ satisfies Axiom 4). Then in the group $G$ Axiom 3) is equivalent to F3).}

Indeed, using the 2-nilpotency of $G$ and Axiom 4) we obtain
$$
(h^{-1}gh)^\alpha = (gg^{-1}h^{-1}gh)^\alpha = (g[g,h])^\alpha = g^\alpha[g,h]^\alpha.
$$
$$
h^{-1}g^\alpha h = g^\alpha g^{-\alpha}h^{-1}g^\alpha h = g^\alpha[g^\alpha,h].
$$
It follows that
\begin{equation} \label{eq:14b}
[g^\alpha,h] = [g,h]^\alpha.
\end{equation}
The converse is also true: Axiom 4) and equality (\ref{eq:14b}) imply Axiom 3) in a 2-nilpotent group.

\medskip
\noindent
{\bf F12)} [{\it  Commutators with inverse scalar}]  Given F1) and F9), for any $g, h \in G$ and invertible $\alpha \in R$, we obtain
\begin{equation}\label{eq:(4b)}
    c(g,h)_{\alpha^{-1}}=c(g^{\alpha^{-1}},h^{\alpha^{-1}})_\alpha^{-\alpha^{-1}}.
\end{equation}

Indeed, as a consequence of F1) and F9) we have
$$e=(g,h)_1=(g,h)_{\alpha^{-1}\alpha} = (g^{\alpha^{-1}},h^{\alpha^{-1}})_\alpha(g,h)^\alpha_{\alpha^{-1}}.
$$
It follows that
$$
(g,h)^{-\alpha}_{\alpha^{-1}}=(g^{\alpha^{-1}},h^{\alpha^{-1}})_\alpha,
$$
and, therefore,

\begin{equation}\label{eq:(4)}
    (g,h)_{\alpha^{-1}}=(g^{\alpha^{-1}},h^{\alpha^{-1}})_\alpha^{-\alpha^{-1}}.
\end{equation}
Replacing $\alpha$-commutators with c-commutators, we obtain (\ref{eq:(4b)}).

\medskip
\noindent
{\bf F13)} [{\it  Symmetry of c-commutators}]  {\it Suppose that the function $(g,\alpha) \to g^\alpha$ satisfies F3). Then in the group $G$ for any $g,h \in G$, $\alpha \in R$ the equality holds}
\begin{equation} \label{eq:sym-c-comm}
c(g,h)_\alpha = c(h,g)_\alpha.
\end{equation}

Indeed, on the one hand
$$
c(g,h)_\alpha = [g,h]^{\binom{\alpha}{2}}h^{-\alpha}g^{-\alpha}(gh)^\alpha.
$$
On the other hand, using F3) and the fact that the center of $G$ is an $R$-ideal (this also follows from F3)), and therefore the $R$-exponents of ordinary commutators lie in $Z(G)$, we obtain
$$
c(h,g)_\alpha = [h,g]^{\binom{\alpha}{2}}g^{-\alpha}h^{-\alpha}(hg)^\alpha = [h,g]^{\binom{\alpha}{2}}g^{-\alpha}h^{-\alpha}(gh[h,g])^\alpha =
$$
$$
[h,g]^{\binom{\alpha}{2}}g^{-\alpha}h^{-\alpha}(gh)^\alpha[h,g]^\alpha = [h,g]^{\binom{\alpha}{2}}h^{-\alpha}g^{-\alpha}[g^{-\alpha},h^{-\alpha}](gh)^\alpha[h,g]^\alpha =
$$
$$
[g,h]^{-\binom{\alpha}{2}}h^{-\alpha}g^{-\alpha}[g,h]^{\alpha^2}(gh)^\alpha[g,h]^{-\alpha} = [g,h]^{-\binom{\alpha}{2} +\alpha^2 -\alpha}h^{-\alpha}g^{-\alpha}(gh)^\alpha.
$$
Since $-\binom{\alpha}{2} +\alpha^2 -\alpha = \binom{\alpha}{2}$, we obtain
(\ref{eq:sym-c-comm}).

\medskip
\noindent
{\bf F14)} [{\it  Associativity of multiplication in a group}]  {\it Suppose that the function $(g,\alpha) \to g^\alpha$ satisfies F3). Then in the group $G$ for any $g,h \in G$, $\alpha \in R$ the following equality holds:}
\begin{equation} \label{eq:ass-c-comm}
 c(gh,f)_\alpha c(g,h)_\alpha =c(g,hf)_\alpha c(h,f)_\alpha.
\end{equation}

Indeed, for any $g,h,f \in G, \alpha \in R$ we have
$$
((gh)f)^\alpha = (g(hf))^\alpha.
$$
Writing both parts in terms of $\alpha$-commutators and exponents of elements $g,h,f$, after the appropriate reductions (using the fact that $(g,h)_\alpha \in Z(G)$, and this follows from F3)) we obtain
\begin{equation} \label{eq:ass}
 (gh,f)_\alpha (g,h)_\alpha =(g,hf)_\alpha  (h,f)_\alpha.
\end{equation}
From where, describing $\alpha$-commutators through the corresponding $c$-commutators and ordinary commutators according to the formula (\ref{eq:def-c-comm}), we obtain (\ref{eq:ass-c-comm}).

\medskip
\noindent
{\bf F15)} [{\it  Axiom 3) $(h^{-1}gh)^\alpha = h^{-1}g^\alpha h$ in terms of c-commutators.}] {\it Suppose that the function $(g,\alpha) \to g^\alpha$ satisfies F0) and F3). Then in the group $G$, Axiom 3) is equivalent to the following identity for $c$-commutators. For any $g,h \in G$, $\alpha \in R$, the following equality holds:}
\begin{equation} \label{eq:conj-c-comm}
   c(h^{-1}g,h)_\alpha =  c(h^{-1},g)_\alpha^{-1}.
\end{equation}

Indeed, let us first assume that Axiom 3) is satisfied in $G$ and write both sides of the equality $(h^{-1}gh)^\alpha = h^{-1}g^\alpha h$ in terms of the exponents of $g$ and $h$, ordinary commutators, and c-commutators. Then, on the one hand, using F3), we obtain
$$
h^{-1}g^\alpha h = g^\alpha g^{-\alpha}h^{-1}g^\alpha h = g^\alpha [g^\alpha,h] = g^\alpha [g,h]^\alpha.
$$
On the other hand, using F0) and F4) (which follows from F3)) we obtain
$$
(h^{-1}gh)^\alpha  = ((h^{-1}g)h)^\alpha = (h^{-1}g)^\alpha h^\alpha (h^{-1}g,h)_\alpha = (gh^{-1}[h^{-1},g])^\alpha h^\alpha (h^{-1}g,h)_\alpha
$$
$$
= g^\alpha h^{-\alpha} (g,h^{-1})_\alpha [h^{-1},g]^\alpha h^\alpha (h^{-1}g,h)_\alpha = g^\alpha [g,h]^\alpha (g,h^{-1})_\alpha  (h^{-1}g,h)_\alpha.
$$
Equating both sides, we obtain
$$
(g,h^{-1})_\alpha  (h^{-1}g,h)_\alpha = e.
$$
Rewriting this equality in terms of $c$-commutators according to the formula (\ref{eq:def-c-comm}) and using F13), which also follows from F3), we obtain (\ref{eq:conj-c-comm}).

Interestingly, the equality (\ref{eq:conj-c-comm}) can be obtained directly by rewriting $ c(h^{-1}g,h)_\alpha$ using F14).

\section{Group structure of tensor completions of finitely generated 2-nilpotent torsion-free groups}\label{se:main}

In this section we describe the algebraic structure of tensor completions $G \otimes_{\N_{2,R}} R$ of finitely generated 2-nilpotent torsion-free groups $G$ over a binomial domain $R$ of characteristic 0. For brevity, we denote $G \otimes R = G \otimes_{\N_{2,R}} R$ and $G\otimes_{\mathcal{H}} R = G\otimes_{\mathcal{H}_2} R$.

The identity map $G \to G$ canonically extends to an embedding $\eta: G \to G\otimes_{\mathcal{H}} R$. Moreover, since $G\otimes_{\mathcal{H}} R$ is an $R$-group, the embedding $\eta$ canonically extends to an $R$-epimorphism $\mu:G\otimes R \to G\otimes_{\mathcal{H}} R$. This epimorphism plays an essential role in our investigations.

Let $\bar u = (u_1, \ldots,u_m)$ and $\bar v = (v_1, \ldots,v_n)$ be bases of the free abelian groups $G/Z(G)$ and $Z(G)$, respectively. Then $\bar u \cdot \bar v = (u_1, \ldots,u_m,v_1, \ldots,v_n)$ is the Mal'tsev basis of $G$. Therefore, each element $g \in G$ can be uniquely represented as a product
$$
g = u_1^{\alpha_1} \ldots u_m^{\alpha_m}v_1^{\beta_1} \ldots v_n^{\beta_n} = \bar u^{\bar\alpha} \cdot \bar v^{\bar \beta},
$$
where $\bar \alpha = (\alpha_1, \ldots,\alpha_m), \bar \beta = (\beta_1, \ldots,\beta_n)$~ are some tuples of integers. By construction, each element of the Hall completion $G \otimes_{\mathcal{H}} R$ has a unique decomposition:

$$
g = u_1^{\alpha_1} \ldots u_m^{\alpha_m}v_1^{\beta_1} \ldots v_n^{\beta_n} = \bar u^{\bar\alpha} \cdot \bar v^{\bar \beta},
$$
where $\bar \alpha = (\alpha_1, \ldots,\alpha_m), \bar \beta = (\beta_1, \ldots,\beta_n)$ are some tuples of elements from $R$. Multiplication and $R$-exponentiation in
$G \otimes_{\mathcal{H}} R$ are defined by the same polynomials as in $G$ with respect to the basis $\bar u \cdot \bar v$.

\begin{lemma}
Let $G$~ be a finitely generated 2-nilpotent torsion-free group, and $R$~ a binomial domain of characteristic zero. Let $\bar u \cdot \bar v$~ be the Mal'tsev basis of $G$, as defined above. Then the set of elements
$$
H_{\bar u \cdot \bar v} = \{ \bar u^{\bar\alpha} \cdot \bar v^{\bar \beta} \mid \bar \alpha \in R^m, \bar \beta \in R^n\}
$$
forms a subgroup (not necessarily an $R$-subgroup) of the tensor $R$-completion $G \otimes R$.
\end{lemma}
\begin{proof}
Let us take two elements $g = \bar u^{\bar\alpha} \cdot \bar v^{\bar \beta}$ and $h = \bar u^{\bar\gamma} \cdot \bar v^{\bar \delta}$ from $H_{\bar u \cdot \bar v}$. Then
$$
gh = u_1^{\alpha_1} \ldots u_m^{\alpha_m}v_1^{\beta_1} \ldots v_n^{\beta_n} u_1^{\gamma_1} \ldots u_m^{\gamma_m}v_1^{\delta_1} \ldots v_n^{\delta_n}.
$$

To show that $gh \in H_{\bar u \cdot \bar v}$, we use the standard collecting process in 2-nilpotent groups, which also works in 2-nilpotent $R$-groups.
Since the elements $v_i^{\beta_i}$ and $v_i^{\delta_i}$ are in the center of $G \otimes R$, then
\begin{equation} \label{eq:H1}
gh = \bar u^{\bar \alpha} \bar u^{\bar \gamma} \bar v^{\bar \beta +\bar \delta}.
\end{equation}

Consider $\bar u^{\bar \alpha} \bar u^{\bar \gamma}$. Using the identity $u_i^{\alpha_i}u_j^{\gamma_j} = u_j^{\gamma_j} u_i^{\alpha_i }[u_i^{\alpha_i },u_j^{\gamma_j}]$, we can rewrite $\bar u^{\bar \alpha} \bar u^{\bar \gamma}$ as follows:
\begin{equation} \label{eq:H2}
\bar u^{\bar \alpha} \bar u^{\bar \gamma} = \bar u^{\bar \alpha +\bar \gamma} \Pi_{i>j}
[u_i^{\alpha_i },u_j^{\gamma_j}].
\end{equation}
Since the identity $[u_i^{\alpha_i },u_j^{\gamma_j}] = [u_i,u_j]^{\alpha_i \gamma_j}$ holds in every 2-nilpotent $R$-group, we can rewrite (\ref{eq:H2}) as
\begin{equation}\label{eq:H3}
\bar u^{\bar \alpha} \bar u^{\bar \gamma} = \bar u^{\bar \alpha +\bar \gamma} \Pi_{i>j}[u_i,u_j]^{\alpha_i \gamma_j}.
\end{equation}
Note that $[u_i,u_j] \in Z(G)$, hence $[u_i,u_j] = \bar v^{\bar k(i,j)}$, where $\bar k(i,j)$ is a tuple from $\Z^n$. Therefore, $[u_i,u_j]^{\alpha_i \gamma_j} = \bar v^{\bar k(i,j)\alpha_i \gamma_j}$, where $\bar k(i,j)\alpha_i \gamma_j$ is a tuple obtained from $\bar k(i,j)$ by scalar multiplication with $\alpha_i \gamma_j$. Substituting this into (\ref{eq:H3}), we obtain
\begin{equation} \label{eq:H4}
\bar u^{\bar \alpha} \bar u^{\bar \gamma} = \bar u^{\bar \alpha +\bar \gamma} \Pi_{i>j}
\bar v^{\bar k(i,j)\alpha_i \gamma_j} = \bar u^{\bar \alpha +\bar \gamma} \bar v^{\Sigma_{i>j}\bar k(i,j)\alpha_i \gamma_j} = \bar u^{\bar \alpha +\bar \gamma} \bar v^{\sigma(\bar \alpha, \bar \beta)},
\end{equation}
where $\sigma(\bar \alpha,\bar \beta)$ is some tuple from $R^m$ depending on $\bar \alpha$ and $\bar \beta$. Now, using (\ref{eq:H4}), we can rewrite (\ref{eq:H1}) as

\begin{equation}\label{eq:H5}
gh =\bar u^{\bar \alpha +\bar \gamma} \bar v^{\sigma(\bar \alpha, \bar \beta)}\bar v^{\bar \beta +\bar \delta} =\bar u^{\bar \alpha +\bar \gamma} \bar v^{\sigma(\bar \alpha, \bar \beta) + \bar \beta +\bar \delta},
\end{equation}
which shows that $H_{\bar u \cdot \bar v}$ is closed under multiplication. To prove closure under inversion, note that
$$
g^{-1} = (u_m^{\alpha_m})^{-1}\ldots (u_1^{\alpha_1})^{-1} = u_m^{-\alpha_m} \ldots u_1^{-\alpha_1}.
$$
The product above belongs to $H_{\bar u \cdot \bar v}$, since $H_{\bar u \cdot \bar v}$ is closed under multiplication.

This proves the lemma.
\end{proof}

\begin{remark}
The subgroup $H_{\bar u \cdot \bar v}$ is not always an $R$-subgroup of $G \otimes R$ (as will be shown below).
\end{remark}

To study tensor completions of $G \otimes R$, we need the following statement.

\begin{theorem} \label{th:tensor-completions}
	Let $G$ be a finitely generated 2-nilpotent torsion-free group, $R$ be a binomial domain of characteristic zero, and $\mu:G \otimes R \to G \otimes_{\mathcal{H}} R$ be the canonical $R$-epimorphism. Let $\bar u \cdot \bar v$ be the Mal'tsev basis of $G$, and $H_{\bar u \cdot \bar v}$ be a subgroup of $G \otimes R$, as above. Then the following holds:
	\begin{itemize}
		\item [1)] The restriction of $\mu$ to $H_{\bar u \cdot \bar v}$ is an isomorphism of abstract groups $H_{\bar u \cdot \bar v}$ onto $G \otimes_{\mathcal{H}} R$. In particular, $G \otimes_{\mathcal{H}} R$ embeds in $G \otimes R$ as an abstract group, and $\mu$ is a retraction to $G \otimes_{\mathcal{H}} R$, viewed as a subgroup $H_{\bar u \cdot \bar v}$ of $G \otimes R$.
		\item [2)] The kernel $D = \ker \mu$ of $\mu$ is an abelian $R$-subgroup of $G \otimes R$ generated as an $R$-group by all c-commutators $c(g,h)_\alpha, g,h \in G \otimes R$, $\alpha \in R$.
		\item [3)] $G \otimes R \simeq (G \otimes_{\mathcal{H}} R) \times D$ as a direct product of abstract groups.
	\end{itemize}
\end{theorem}

\begin{proof}
	Note that the $R$-epimorphism $\mu: G\otimes R \to G \otimes_{\mathcal{H}} R$ is the identity on $G$, hence it maps $\bar u^{\bar\alpha} \cdot \bar v^{\bar \beta}$ of $G \otimes R$ to $\bar u^{\bar\alpha} \cdot \bar v^{\bar \beta}$ of $G \otimes_{\mathcal{H}} R$. Consequently, the restriction of $\mu$ to $H_{\bar u \cdot \bar v}$ is an isomorphism of abstract groups $H_{\bar u \cdot \bar v} \to G \otimes_{\mathcal{H}} R$. This proves 1).

	Therefore, on the one hand
	$$
	\mu((xy)^\lambda) = (\mu(x)\mu(y))^\lambda= \mu(x)^\lambda \mu(y)^\lambda [\mu(x),\mu(y)]^{-\binom{\lambda}{2}}.
	$$
	On the other hand
	$$
	\mu((xy)^\lambda) = \mu(x^\lambda y^\lambda [x,y]^{-\binom{\lambda}{2}} c(x,y)_\lambda)= \mu(x)^\lambda \mu(y)^\lambda [\mu(x),\mu(y)]^{-\binom{\lambda}{2}} \mu(c(x,y)_\lambda).
	$$
	It follows that $\mu(c(x,y)_\lambda) = e$, so $c(x,y)_\lambda \in \ker \mu$ for any $x,y \in G\otimes R$ and $\lambda \in R$.
    Denote by $D$ the $R$-subgroup of $G \otimes R$ generated by all c-commutators $c(x,y)_\lambda$, i.e.
	$$
	D = \langle c(x,y)_\lambda \mid x,y \in G\otimes R, \lambda \in R\rangle_R.
	$$
	It turns out that $D$ is a central $R$-subgroup of $G \otimes R$ and $D \leq \ker \mu$.
    We claim that $D = \ker \mu$ and $G\otimes R = H_{\bar u \cdot \bar v} \times D$ as abstract groups. Since $\mu$ is a retract onto $H_{\bar u \cdot \bar v}$ and $D \leq Z(G\otimes R)$, it suffices to show that $G\otimes R = H_{\bar u \cdot \bar v}\cdot D$.
	
	Now note that the group $G \otimes R$ \ $R$-generated by the group $G$. Therefore, each element $g \in G \otimes R$ can be represented as
	$$
	g = w(g_1, \ldots,g_k),
	$$
	where $w(z_1, \ldots,z_k)$ is an $R$-word and $g_1, \ldots,g_k \in G$. We prove by induction on the number $e(w)$ of exponentiations in $w$ that $g = h\cdot d$, where $h \in H_{\bar u \cdot \bar v}$ and $d \in D$. If $e(w) = 0$, then $g \in G \leq H_{\bar u \cdot \bar v}$. Since $H_{\bar u \cdot \bar v}\cdot D$ is closed under multiplication, it suffices to consider the case when
	$$
	g = (f_1^{\gamma_1} \ldots f_k^{\gamma_k})^\lambda,
	$$
	where $f_1, \ldots, f_k\in G \otimes R$ and $\gamma_1, \ldots, \gamma_k, \lambda \in R$.
    By induction, $f_i^{\alpha_i} \in H_{\bar u \cdot \bar v} \cdot D$, so their product also belongs to $H_{\bar u \cdot \bar v} \cdot D$. Therefore, we can assume that
	$$
	g = (\bar u^{\bar\alpha} \bar v^{\bar \beta} d)^\lambda = (\bar u^{\bar\alpha})^\lambda \bar v^{\bar \beta \lambda} d^\lambda,
	$$
	where $d \in D$, and $\bar \beta \lambda$ is the scalar product of $\bar \beta$ and $ \lambda$. Therefore, it suffices to show that $(\bar u^{\bar\alpha})^\lambda = (u_1^{\alpha_1}\ldots u_m^{\alpha_m})^\lambda\in H_{\bar u \cdot \bar v}\cdot D$. We use induction on $m$. Denote $f = u_1^{\alpha_1}\ldots u_{m-1}^{\alpha_{m-1}}$. Then
	$$
	(u_1^{\alpha_1}\ldots u_m^{\alpha_m})^\lambda = f^\lambda (u_m^{\alpha_m})^\lambda [f,u_m^{\alpha_m}]^{-\binom{\lambda}{2}} c(f,u_m^{\alpha_m})_\lambda.
	$$
By induction, $f^\lambda \in H_{\bar u \cdot \bar v} \cdot D$. Obviously, $(u_m^{\alpha_m})^\lambda = u_m^{\alpha_m\lambda} \in H_{\bar u \cdot \bar v}$ and $ c(f,u_m^{\alpha_m})_\lambda \in D$. Note that
$$
[f,u_m^{\alpha_m}]^{-\binom{\lambda}{2}} = [f,u_m]^{-\binom{\lambda}{2}\alpha_m}
$$
and
$$
[f,u_m] = [u_1^{\alpha_1}\ldots u_{m-1}^{\alpha_{m-1}},u_m] = \prod_{i=1}^{m-1} [u_i,u_m]^{\alpha_i},
$$
from which it follows that
$$
[f,u_m]^{-\binom{\lambda}{2}\alpha_m} = \left(\prod_{i=1}^{m-1} [u_i,u_m]^{\alpha_i}\right)^{-\binom{\lambda}{2}\alpha_m} = \prod_{i=1}^{m-1} [u_i,u_m]^{-\alpha_i\binom{\lambda}{2}\alpha_m}.
$$
Since $[u_i,u_m] \in Z(G)$, then
$$
[u_i,u_m]^{-\alpha_i\binom{\lambda}{2}\alpha_m} = \bar v^{\bar \delta_{i,j}}
$$
for some tuple $\bar \delta_{i,j} \in R^n$. Therefore,
$$
\prod_{i=1}^{m-1} [u_i,u_m]^{-\alpha_i\binom{\lambda}{2}\alpha_m} = \prod_{i=1}^{m} \bar v^{\bar \delta_{i,j}} = \bar v^{\delta}
$$
for some $\delta \in R^n$. This proves that $[f,u_m^{\alpha_m}]^{-\binom{\lambda}{2}} \in H_{\bar u \cdot \bar v}$ and, consequently, that $(\bar u^{\bar\alpha})^\lambda = (u_1^{\alpha_1}\ldots u_m^{\alpha_m})^\lambda\in H_{\bar u \cdot \bar v}\cdot D$. We obtain that $g = (f_1^{\gamma_1} \ldots f_k^{\gamma_k})^\lambda \in H_{\bar u \cdot \bar v}\cdot D$, and hence by induction $g = w(g_1, \ldots,g_k) \in H_{\bar u \cdot \bar v}\cdot D$. Consequently, $G \otimes R =H_{\bar u \cdot \bar v}\cdot D$ and $D = \ker \mu$. Whence
$H_{\bar u \cdot \bar v} \cap D = \{e\}$ and $G \otimes R =H_{\bar u \cdot \bar v}\times D$. This proves the theorem.
\end{proof}

%%%%%%%%%%%%%%%%%%%%%%%

\section{On the \texorpdfstring{$R$}{}-structure of tensor \texorpdfstring{$R$}{}-completions of torsion-free nilpotent groups}
\label{sec:6}

In this section, we provide a general description of tensor completions $G\otimes R$ for arbitrary 2-nilpotent finitely generated torsion-free groups $G$ over arbitrary binomial integral domains $R$. This description can be significantly refined for specific groups $G$ and rings $R$. We do this for free 2-nilpotent groups of rank 2 and special rings in the next section. The important point here is that the general approach, although clear, is technically challenging; its full implementation depends on the given group $G$ and the given ring $R$, which is beyond the scope of this paper.

Let $G$ be an arbitrary 2-nilpotent finitely generated torsion-free group and $R$ be an arbitrary binomial domain. By Theorem \ref{th:tensor-completions}, we know the group structure of $G\otimes R$, namely,
$$
G\otimes R \simeq (G\otimes_{\mathcal{H}} R) \times D,
$$
where $G\otimes_{\mathcal{H}} R$ is the Hall completion of $G$ over $R$, and $D$ is the abelian $R$-subgroup ($R$-module) of the $R$-group $G\otimes R$ generated by all $c$-commutators in $G\otimes R$. Thus, we know the multiplication in $G\otimes R$. Now we need to understand what the $R$-exponentiation is in the group $G\otimes R$.

To define the $R$-exponentiation on an arbitrary element $g \in G\otimes R$, we decompose $g$ into a product $g = h d$, where $h \in G\otimes_{\mathcal{H}} R, d \in D$. Since $D$ is a central $R$-subgroup of $G\otimes R$, for any $\alpha \in R$, we must have $(hd)^\alpha = h^\alpha d^\alpha$. Therefore, to define the $R$-exponentiation on $G\otimes R$, it suffices to define it on the subgroups $G\otimes_{\mathcal{H}} R$ and $D$.

\medskip
\noindent
{\bf $R$-exponentiation on $D$.} Since $D$ is an abelian (even central) $R$-subgroup of $G\otimes R$, the $R$-exponentiation in $D$ is defined as in an $R$-module. Therefore, to understand how $R$ acts on $D$, we need to understand the structure of the $R$-module $D$. To do this, it suffices to take the set of generators $C= \{c(g,h)_\alpha \mid g, h \in G\otimes R, \alpha \in R\}$ of the module $D$, find the set of defining relations $S$ for $C$ in $D$, and take the quotient module $L/S_R$ of the free $R$-module $L = \bigoplus\limits_{c \in C} cR$ with basis $C$ by the $R$-submodule $S_R = \langle S\rangle_R$ of module $L$, $R$-generated by $S$. Clearly, $D \simeq L/S_R$ and the $R$-exponentiation in $D$ is induced from the $R$-module $L$. As a set of defining relations for $S$, we can take all the relations on $c$-commutators described in section \ref{sec:commutators} and prove that they are defining as we construct an $R$-exponentiation on $G \otimes R$.

\medskip
\noindent
{\bf $R$-exponentiation on $G \otimes_{\mathcal{H}} R$.} To define the $R$-exponentiation on the subgroup $G\otimes_{\mathcal{H}} R$, we have a formula (\ref{eq:exp-c-comm}) from section \ref{sec:commutators}, according to which for $x, y \in G\otimes_{\mathcal{H}} R, \alpha \in R$ we have:
$$
(xy)^\alpha = x^\alpha y^\alpha [x,y]^{-\binom{\alpha}{2}} c(x,y)_\alpha.
$$
For convenience, we generalize this formula for the product of arbitrary $n$ elements $x_1, \ldots, x_n$ from the group $G\otimes_{\mathcal{H}} R$:
\begin{equation} \label{eq:general-c}
(x_1 \ldots x_n)^\alpha = x_1^\alpha \ldots x_n^\alpha \tau_2(x_1, \ldots,x_n)^{-\binom{\alpha}{2}} c(x_1, \ldots,x_n)_\alpha,
\end{equation}
where $(x_1 \ldots x_n)^\alpha = x_1^\alpha \ldots x_n^\alpha \tau_2(x_1, \ldots,x_n)^{-\binom{\alpha}{2}} $ is the Hall formula for exponentiation in the Hall group $G \otimes_{\mathcal{H}} R$ (here $\tau_2(x_1, \ldots,x_n)$ is the corresponding Petresco word from Axiom 5 of $R$-exponential groups, see Section \ref{se:2.1}), and
$c(x_1, \ldots,x_n)_\alpha$ is the unique element from $D$ satisfying this equality.

By performing the $R$-exponentiation in (\ref{eq:general-c}) step by step, we can rewrite this formula as follows:
$$
(x_1 \ldots x_n)^\alpha = ((x_1 \ldots x_{n-1})x_n)^\alpha = (x_1 \ldots x_{n-1})^\alpha x_n^\alpha [x_1 \ldots x_{n-1},x_n]^{-\binom{\alpha}{2}} c(x_1 \ldots x_{n-1},x_n)_\alpha.
$$
Now, expanding in the same way,
$$
(x_1 \ldots x_{n-1})^\alpha = (x_1 \ldots x_{n-2})^\alpha x_{n-1}^\alpha [x_1 \ldots x_{n-2},x_{n-1}]^{-\binom{\alpha}{2}} c(x_1 \ldots x_{n-2},x_{n-1})_\alpha
$$
and then by induction, we obtain
\begin{equation} \label{eq:general-formula}
(x_1 \ldots x_n)^\alpha = \Pi_{i = 1}^n x_i^\alpha \Pi_{i = 1}^{n-1}[x_1\ldots x_{i},x_{i+1}]^{-\binom{\alpha}{2}} \Pi_{i = 1}^{n-1} c(x_1\ldots x_i,x_{i+1})_\alpha.
\end{equation}
Hence, $\tau_2(x_1, \ldots,x_n) = \Pi_{i = 1}^{n-1}[x_1\ldots x_{i},x_{i+1}]$, which is well known, and
\begin{equation} \label{eq:long-c}
c(x_1, \ldots,x_n)_\alpha = \Pi_{i = 1}^{n-1} c(x_1\ldots x_i,x_{i+1})_\alpha,
\end{equation}
which we will use below.

Note that since multiplication in $G \otimes R$ is associative, we can calculate the $\alpha$-exponent of the product $(x_1 \ldots x_n)^\alpha$ by placing the parentheses differently, but the result will be the same. This yields one type of relation on $c$-commutators. For $n = 3$, these relations are written out in the formula (\ref{eq:ass-c-comm}):
$$
c(gh,f)_\alpha=c(g,hf)_\alpha c(g,h)_\alpha^{-1} c(h,f)_\alpha .
$$
Clearly, using this formula, we can obtain other relations of this type. Furthermore, defining the $R$-exponentiation on $G\otimes R$ via elements of the type $(x_1 \ldots x_n)^\alpha$, it is necessary to ensure that the exponentiation is well-defined since the same element in a group  $G \otimes_{\mathcal{H}} R$ can still be represented by different products $x_1 \ldots x_n$.  To ensure that the $R$-exponentiation is defined correctly, one can either impose new relations on the $c$-commutators each time or, which is much more convenient, represent elements of $G\otimes_{\mathcal{H}} R$ in advance in a unique form through the chosen generators. To this end, recall that the subgroup $G\otimes_{\mathcal{H}} R$ appears in $G$ as the set of all products of the form
\begin{equation}
   g = u_1^{\alpha_1} \ldots u_m^{\alpha_m}v_1^{\beta_1} \ldots v_n^{\beta_n}, 
\end{equation}
where $(u_1, \ldots, u_m)$ is a basis of the free abelian group $G/Z(G)$, $(v_1, \ldots,v_n)$ is a basis of the free abelian group $Z(G)$, and $\alpha_1, \ldots,\alpha_m, \beta_1, \ldots,\beta_n$ are arbitrary elements of $R$ (see section \ref{se:main}). This product is unique for each element $g \in G$. Thus, it suffices for us to define in the group $G\otimes R$ exponents of the form
$$
 (u_1^{\alpha_1} \ldots u_m^{\alpha_m}v_1^{\beta_1} \ldots v_n^{\beta_n})^\alpha.
$$
Since $v_i^{\beta_i} \in Z(G)$, as before
$$
 (u_1^{\alpha_1} \ldots u_m^{\alpha_m}v_1^{\beta_1} \ldots v_n^{\beta_n})^\alpha = (u_1^{\alpha_1} \ldots u_m^{\alpha_m})^\alpha v_1^{\beta_1\alpha} \ldots v_n^{\beta_n\alpha},
$$
that is, in the end, we only need to determine the exponents of the form
$$
(u_1^{\alpha_1} \ldots u_m^{\alpha_m})^\alpha,
$$
which is what we do according to the formula (\ref{eq:general-formula}), where $x_i = u_i^{\alpha_i}$, and we set $(u_i^{\alpha_i})^\alpha = u_i^{\alpha_i\alpha}$. So, raising to the exponent $\alpha \in R$ for an element
$$
g=  u_1^{\alpha_1} \ldots u_m^{\alpha_m}v_1^{\beta_1} \ldots v_n^{\beta_n}d,
$$ 
where $\alpha_i, \beta_j \in R$, $d \in D$, is determined by the formula
\begin{equation} \label{eq:main-def}
(u_1^{\alpha_1} \ldots u_m^{\alpha_m}v_1^{\beta_1} \ldots \ldots v_n^{\beta_n}d)^\alpha =     u_1^{\alpha_1\alpha} \ldots u_m^{\alpha_m \alpha}v_1^{\beta_1\alpha+\sigma_1} \ldots v_n^{\beta_n\alpha +\sigma_n}d^\alpha c(u_1^{\alpha_1}, \ldots,u_m^{\alpha_m})_\alpha,
\end{equation}
where 
\begin{equation} \label{eq:main-def-2}
    v_1^{\sigma_1} \ldots v_n^{\sigma_n} = \tau_2(u_1^{\alpha_1}, \ldots,u_m^{\alpha_m})^{-\binom{\alpha}{2}},
\end{equation}
and, according to the formula (\ref{eq:long-c}),
\begin{equation} \label{eq:main-def-3}
    c(u_1^{\alpha_1}, \ldots,u_m^{\alpha_m})_\alpha = \Pi_{i = 1}^{n-1} c(u_1^{\alpha_1}\ldots u_i^{\alpha_i},u_{i+1}^{\alpha_{i+1}})_\alpha.
\end{equation}
Let us denote by $Exp: (G \otimes R) \times R \to G \otimes R$ the exponentiation introduced in this way.
%При этом добавляя к определяющим соотношениям $D$ соответствующие соотношения на c-коммутаторы, полученные из ассоциативности, как объяснялось выше. 

Now we have to ensure that $Exp$ satisfies all Axioms 1)-4) of $R$-groups. To do this, we add to the defining relations of $D$ all relations from F0)-F15) on c-commutators from section \ref{sec:commutators}, in particular, all relations that correspond to Axioms 1)-4). The result is an $R$-group, denoted by $H$, which, as an abstract group, is isomorphic to $G \otimes R$; contains $G$ as a subgroup, and is generated by it as an $R$-group. Moreover, all $R$-relations ($R$-words) that hold in $H$ also hold in the $R$-group $G \otimes R$, which follows from the construction of the $R$-exponentiation in $H$. Thus, we have an $R$-epimorphism $H \to G\otimes R$. From the universal properties of the tensor completion $G \otimes R$, it follows that there is an inverse  $R$-epimorphism $G \otimes R \to H$, and therefore $H$ and $G \otimes R$ are $R$-isomorphic.

Note that some relations obtained from F0)-F15) are not independent, as they are consequences of the others. However, this does not affect the final structure of the module $D$, and in some cases, they simplify the proofs. We will now make some useful remarks about the added relations. To do this, we will use formulas F0)-F15) from section \ref{sec:commutators} and, for convenience, we will decompose the entire construction into appropriate steps E0) - E15) below, following the strategy described above.

\medskip
\noindent
Е0) Note that by construction (see (\ref{eq:main-def})) $Exp$ satisfies property F0).

\medskip
\noindent
E1) To ensure F1), we set $c(g,h)_r = e$ for all $c$-commutators $c(g,h)_r$ from $C$ such that $r \in \mathbb{Q}\cap R$.

\medskip
\noindent
E3) $Exp$ satisfies F3) by construction.

\medskip
\noindent
E4) $Exp$ satisfies F4) by construction.

\medskip
\noindent
E5) $Exp$ satisfies F5) by construction.

\medskip
\noindent
E6) $Exp$ satisfies F6) since $Exp$ satisfies F0) and F3).

\medskip
\noindent
E0) + E3)  From the satisfiability of conditions F0) and F3) it follows that for any $g, h, g_1, h_1 \in G\otimes R, \lambda \in R$ if $gh = g_1h_1$ modulo the commutator subgroup $H$ of the group $G\otimes R$ (we write this as $gh \underset{H}{=} g_1h_1$), then
$$
c(g,h)_\lambda =c(g_1,h_1)_\lambda.
$$
Indeed, on the one hand
$$
(gh)^\lambda \underset{H}{=} g^\lambda h^\lambda c(g,h)_\lambda,
$$
and on the other hand
$$
(gh)^\lambda \underset{H}{=} (g_1h_1)^\lambda \underset{H}{=} g_1^\lambda h_1^\lambda c(g_1,h_1)_\lambda.
$$
From where (see the argument in the proof of E10' below) we obtain what is required.  

\medskip
\noindent
Е1) + Е6) Note that E1) and E6) guarantee that $Exp$ satisfies axioms 1) of an exponential $R$-group.

\medskip
\noindent
E7) We impose the relations from F7): 
$$
c(f,h)_{\alpha}^{\beta}c(f^{\alpha},h^{\alpha})_{\beta}
    =c(f,h)_{\beta}^{\alpha}c(f^{\beta},h^{\beta})_{\alpha}
$$
for any $f,h \in G$ and $\alpha,\beta\in R$, which follow from the commutativity of $R$ and the satisfiability of F0), F3), and Axiom 2.2. Clearly, these relations follow from E0), E3), and E9) below (which describe Axiom 2.2). 

\medskip
\noindent
E8) To guarantee Axiom 2.1, $g^{\alpha+\beta} = g^{\alpha}g^{\beta}$, we need to add the relations
$$
 c(g,h)_{\alpha +\beta} = c(g,h)_{\alpha}c(g,h)_{\beta}
$$
for all $c$-commutators from $C$. Then, provided that F3) holds for $Exp$, we have, according to F8), that $Exp$ satisfies Axiom 2.1.

\medskip
\noindent
E9) To guarantee Axiom 2.2, $g^{\alpha \beta} = (g^{\alpha})^{\beta}$, we need to add the relations 
 $$
 c(g,h)_{\alpha\beta} = c(g^{\alpha},h^{\alpha})_\beta  c(g,h)_{\alpha}^{\beta}.
$$
Then, provided that F3) holds for $Exp$, we have, according to F9), that $Exp$ satisfies Axiom 2.2.

\medskip
\noindent
E10) To guarantee Axiom 4), we need to add the relations from F10), namely 
$$
c(g,h)_\alpha = e
$$
for all $g, h \in G$ such that $[g,h] = e$, and all $\alpha \in R$.

It will be convenient for us to rewrite the relations from E10) in a slightly different form. Denote $\bar \alpha = (\alpha_1, \ldots,\alpha_m), \bar \beta = (\beta_1, \ldots,\beta_m)$, $\bar u^{\bar \alpha} = u_1^{\alpha_1} \ldots u_m^{\alpha_m}, \bar u^{\bar \beta} = u_1^{\beta_1} \ldots u_m^{\beta_m}$. Then, modulo the commutator subgroup $H = (G\otimes R)'$, we have for any $\lambda \in R$
$$
(\bar u^{\bar \alpha} \bar u^{\bar \beta})^\lambda  \underset{H}{=} (\bar u^{\bar \alpha})^\lambda (\bar u^{\bar \beta})^\lambda c(\bar u^{\bar \alpha},\bar u^{\bar \beta})_\lambda \underset{H}{=} \bar u^{{\bar \alpha}\lambda } \bar u^{{\bar \beta}\lambda }c(\bar u^{\bar \alpha})_\lambda c(\bar u^{\bar \beta})_\lambda c(\bar u^{\bar \alpha}, \bar u^{\bar \beta})_\lambda, 
$$
where $c(\bar u^{\bar \alpha})_\lambda = c(u_1^{\alpha_1}, \ldots, u_m^{\alpha_m})_\lambda$ and $c(\bar u^{\bar \beta})_\lambda = c(u_1^{\beta_1}, \ldots , u_m^{\beta_m})_\lambda$. 
On the other side, $\bar u^{\bar \alpha} \bar u^{\bar \beta} \underset{H}{=} \bar u^{\bar \alpha +\bar \beta}$. Therefore
$$
(\bar u^{\bar \alpha} \bar u^{\bar \beta})^\lambda  \underset{H}{=} (\bar u^{\bar \alpha +\bar \beta})^\lambda \underset{H}{=} \bar u^{(\bar \alpha +\bar \beta)\lambda} c(\bar u^{\bar \alpha +\bar \beta})_\lambda.
$$
From here, since $G\otimes R = (G\otimes_{\mathcal{H}} R) \times D$ and $H \leq G\otimes_{\mathcal{H}} R$, we obtain that
$$
c(\bar u^{\bar \alpha})_\lambda c(\bar u^{\bar \beta})_\lambda c(\bar u^{\bar \alpha}, \bar u^{\bar \beta})_\lambda = c(\bar u^{\bar \alpha +\bar \beta})_\lambda.
$$
If $[\bar u^{\bar \alpha}, \bar u^{\bar \beta}] = e$ then, according to E10), $c(\bar u^{\bar \alpha}, \bar u^{\bar \beta})_\lambda =e$ and therefore the following relations are satisfied.

\medskip
\noindent
Е10')  If $[\bar u^{\bar \alpha}, \bar u^{\bar \beta}] = e$, then 
\begin{equation} \label{eq:E11)}
c(\bar u^{\bar \alpha +\bar \beta})_\lambda =  c(\bar u^{\bar \alpha})_\lambda c(\bar u^{\bar \beta})_\lambda .   
\end{equation}
It is clear that E10') implies E10), and therefore Axiom 4).

\medskip
\noindent
Е11) As noted in F11), Axiom 4) and property F3) imply (are equivalent to) the satisfiability of Axiom 3) of $R$-exponential groups.

\medskip
\noindent
E12) We impose the relations from F12):
$$
c(g,h)_{\alpha^{-1}}=c(g^{\alpha^{-1}},h^{\alpha^{-1}})_\alpha^{-\alpha^{-1}}
$$
for any $g, h \in G$ and invertible $\alpha \in R$. As noted in F12), these relations depend on the others; they follow from F1) and F9), but it is convenient to mention them separately and use them in the proofs.

\medskip
\noindent
E13) Relations from F13): for any $g,h \in G$, $\alpha \in R$
$$
c(g,h)_\alpha = c(h,g)_\alpha.
$$
They follow from F3), but it is convenient to have them directly.

\medskip
\noindent
E14) Relations from F14): for any $g,h \in G$, $\alpha \in R$
$$
 c(gh,f)_\alpha c(g,h)_\alpha =c(g,hf)_\alpha c(h,f)_\alpha.
$$
They follow from F3), but it is convenient to have them directly.

\medskip
\noindent
E15) Relations from F15): for any $g,h \in G$, $\alpha \in R$
$$
   c(h^{-1}g,h)_\alpha =  c(h^{-1},g)_\alpha^{-1}.
   $$
These relations are equivalent to Axiom 3) in the presence of F0) and F3), so they are derived from E0), E3) and E10). But they are convenient to work with, and we add them.

\begin{theorem} \label{th:5.1}
Let $G$ be an arbitrary 2-nilpotent finitely generated torsion-free group and $R$ an arbitrary binomial domain. Then the $R$-group $G\otimes R$ satisfies the following conditions:
\begin{itemize}
    \item [1)] $G \otimes R$, as an abstract group (not an $R$-group!), is isomorphic to $(G \otimes_{\mathcal{H}} R) \times D$, where $G \otimes_{\mathcal{H}} R$ is the Hall $R$-completion of $G$, considered as an abstract group (not an $R$-group), and $D$ is the $R$-module generated by all $c$-commutators $c(x^\alpha,y^\beta)_\mu, \alpha, \beta, \mu \in R$, with defining relations E0)-E15);
    \item [2)] $R$-exponentiation  $Exp: (G \otimes R) \times R \to G \otimes R$ in the group $G \otimes R$ is defined by formulas (\ref{eq:main-def}), (\ref{eq:main-def-2}) and (\ref{eq:main-def-3});
    \item [3)] $D$ is a free $R$-module.
\end{itemize} 
\end{theorem}
\begin{proof}
    Statements 1) and 2) follow from Theorem \ref{th:tensor-completions} and the arguments presented above in the discussion of relations E0)-E15). Now we prove 3). By construction, the $R$-module $D$ has a presentation given by generators $C= \{c(g,h)_\alpha \mid g, h \in G\otimes R, \alpha \in R\}$ and defining relations $S$ obtained by the union of all relations E0)-E15). Note that all relations from $S$ have the form when one c-commutator is equal to the product of other c-commutators or their $R$-exponents. Indeed, all relations from E0)-E15) are written in this way, except for relations from E7) and E14). Relations from E7) have the form 
    $$
c(f,h)_{\alpha}^{\beta}c(f^{\alpha},h^{\alpha})_{\beta}
    =c(f,h)_{\beta}^{\alpha}c(f^{\beta},h^{\beta})_{\alpha},
$$
and therefore they can be rewritten in the form
 $$
c(f^{\alpha},h^{\alpha})_{\beta}
=c(f,h)_{\beta}^{\alpha}c(f^{\beta},h^{\beta})_{\alpha}c(f,h)_{\alpha}^{-\beta}.
$$
The relations from E14) are of the form
    $$
 c(gh,f)_\alpha c(g,h)_\alpha =c(g,hf)_\alpha c(h,f)_\alpha,
$$
    and therefore they can be rewritten in the form
    $$
 c(gh,f)_\alpha  =c(g,hf)_\alpha c(h,f)_\alpha c(g,h)_\alpha^{-1}.
$$
    
    Imposing such relations is equivalent to removing the c-commutators that appear on the left-hand sides of the relations from the set of generators and the corresponding relations from the set of defining relations. If we do this in the proper order, we end up with an R-module with an empty set of defining relations; that is, a free R-module.
\end{proof}

\begin{remark}
    By Theorem \ref{th:5.1}, the $R$-module $D$ is a free $R$-module. However, finding a basis for this module is, in general, a rather difficult problem. Indeed, firstly, Axioms 2.1 and 2.2 allow us to write c-commutators of the form $c(g,h)_\lambda$ in terms of a product (possibly with $R$-exponents) of c-commutators of the form $c(g_1,h_1)_\mu$, where $g_1,h_1 \in G\otimes R$ and $\mu$ run over elements from an arbitrary but fixed set of generators of the ring $R$. Secondly, Axiom 4) depends on the commutation of elements in the group $G$; moreover, each pair of commuting elements $g,h \in G\otimes R$ produces a whole series of relations from E10') of the form
    $$
    c(\bar u^{\bar \alpha +\bar \beta})_\lambda =  c(\bar u^{\bar \alpha})_\lambda c(\bar u^{\bar \beta})_\lambda.
    $$
    This series of relations depends essentially on the  sets of additive generators of  the ring $R$. Thus, the structure of the module $D$ depends on both the group $G$ and the ring $R$. 
In the next section, we give two examples of the structure of the module $D$ when the group $G$ is the group $UT_3(\mathbb Z)$, that is, the free 2-nilpotent group of rank 2, and the ring $R$ is either the polynomial ring $\mathbb{Q}[t]$ or the field of rational functions $\mathbb{Q}(t)$.
\end{remark}

\section{Free 2-step nilpotent \texorpdfstring{$R$}{}-groups of rank 2}

In this section, we construct free 2-step nilpotent $R$-groups of rank 2 when the ring $R$ is either the polynomial ring $R=\mathbb{Q}[t]$ or its field of fractions $R=\mathbb{Q}(t)$. As noted in section \ref{se:tensor-completions}, every free 2-step nilpotent $R$-group is an $N\otimes R$ tensor completion of the ordinary free 2-nilpotent group $N$ of the corresponding rank. Therefore, we can use the results from the previous section on tensor completions $N \otimes R$.

For the rest of this section, we denote by $N$ the free nilpotent group of rank 2 generated by two elements $x$ and $y$, and by $N \otimes_{\mathcal{H}} R$ its Hall $R$-completion. The tuple $u = (x,y,[y,x])$ constitutes a Mal'tsev basis for both $N$ and $N \otimes_{\mathcal{H}} R$. 

Now we give a precise description of the group $G = N \otimes R$ according to the general description of groups of the form $G \otimes R$ in Theorem \ref{th:5.1} of Section \ref{sec:6}. By Theorem \ref{th:5.1} 
$$
G = N\otimes R \simeq (N \otimes_{\mathcal{H}} R) \times D,
$$
where $D$ is the $R$-module generated by all c-commutators $c(g,h)_\lambda$ of the group $N \otimes R$. Exponentiation in the group $N \otimes R$ is given by the formula
\begin{equation} \label{eq:exponentiation-new}
	  (x^\alpha y^\beta [y,x]^\gamma)^\mu = x^{\alpha\mu} y^{\beta\mu} [y,x]^{-\alpha \beta\binom{\mu}{2} +\gamma\mu} c(x^\alpha,y^\beta)_\mu.  
\end{equation}
Note that $D = \langle c(x^\alpha,y^\beta)_\lambda \mid \alpha, \beta, \lambda \in R \rangle_R$. Indeed, by E0) + E3), if $gh = g_1h_1$ modulo the commutator subgroup $H$ of the group $N\otimes R$, then $c(g,h)_\lambda =c(g_1,h_1)_\lambda$. Since for any $g,h \in N\otimes_{\mathcal{H}} R$ there exist $\alpha, \beta \in R$ such that $gh \underset{H}{=} x^\alpha y^\beta$, we obtain the desired result.

Our task is to establish the exact structure of the $R$-module $D$, following the description from Theorem \ref{th:5.1}.

\medskip
\noindent
Case 1.  $R= \mathbb{Q}[t]$.

The relations in E1), E8) and E9) show that the $R$-module $D$ is generated by the elements $c(x^\alpha,y^\beta)_t$, where $\alpha, \beta \in R$, so that 
\begin{equation} \label{eq:gensD}
    D = \langle c(x^\alpha,y^\beta)_t \mid \alpha, \beta \in R\rangle_R.
\end{equation}
However, this is not yet a minimal set of generators of the $R$-module $D$, since Axiom 4) makes the c-commutators $c(g,h)_\lambda$ trivial, provided that $[g,h] = e$. For example, $c(x^\alpha,x^\beta)_\lambda = e$, and also $c(y^\alpha,y^\beta)_\lambda = e$. We claim that $D$ is a free $R$-module and now describe its basis precisely. For this, we will use formula (\ref{eq:E11)}) from E10'), namely, if $[x^\alpha y^\beta,x^{\alpha_1} y^{\beta_1}] =e$, then
\begin{equation} \label{eq:additive}
c(x^{\alpha+\alpha_1},y^{\beta+\beta_1})_\lambda = c(x^{\alpha},y^{\beta})_\lambda c(x^{\alpha_1},y^{\beta_1})_\lambda.
\end{equation}
 We introduce an equivalence relation on $R^2$ by setting $(\alpha_1,\beta_1)\sim (\alpha_2,\beta_2)$ if and only if the vectors $(\alpha_1,\beta_1)$ and $(\alpha_2,\beta_2)$ are linearly dependent over $R$ (the determinant $\begin{vmatrix} \alpha_1 & \beta_1 \\ \alpha_2 & \beta_2 \end{vmatrix}=0$). We denote by $(\overline{\alpha,\beta})$ the equivalence class containing the pair $(\alpha,\beta)$. It is known that the centralizer of a non-central element $x^\alpha y^\beta$ in the group $N\otimes_{\mathcal{H}} R$ consists of the center of $N\otimes_{\mathcal{H}} R$ and all elements $x^{\alpha_1} y^{\beta_1}$ such that $(\alpha,\beta) \sim (\alpha_1,\beta_1)$.
Since $c(x^\alpha,y^\beta)_\lambda = e$ if $\alpha = 0$ or $\beta =0$, then in constructing a basis for $D$ we can restrict ourselves to the elements $c(x^\alpha,y^\beta)_\lambda$, where $\alpha \neq 0$ and $\beta \neq 0$. Let $R_0^2=\{(\alpha,\beta)\mid\;\alpha,\beta\in R,\;\alpha\beta\neq 0\}$. From another point of view, equivalence classes are maximal one-generated submodules of the module $R \times R$ without zero element. The pair $(\alpha,\beta) \in R_0^2$ generates the class $(\overline{\alpha,\beta})$ as a submodule if and only if the greatest common divisor of the pair $(\alpha,\beta)$ is equal to $1$. Such a pair $(\alpha,\beta)$ will be called a \emph{special representative} for $(\overline{\alpha,\beta})$.

Let the pair $(\alpha, \beta) \in R_0^2$ be a special representative in the equivalence class $p = (\overline{\alpha,\beta})$.
If $(\alpha_1, \beta_1) \in (\overline{\alpha,\beta})$, then there exists $\gamma$ in $R$ such that $\alpha_1 = \gamma \alpha$ and $\beta_1 = \gamma \beta$. Suppose that $\gamma = \gamma_1 + \gamma_2$ in the ring $R$. Then 
$$
[x^{\gamma_1 \alpha }y^{\gamma_1 \beta }, x^{\gamma_2 \alpha }y^{\gamma_2 \beta }] = [(x^{\alpha}y^{\beta})^{\gamma_1}, (x^{\alpha }y^{\beta})^{\gamma_2}] = [(x^{\alpha}y^{\beta}), (x^{\alpha }y^{\beta})]^{\gamma_1\gamma_2} = e,
$$
and therefore, according to the formula (\ref{eq:additive}), 
$$
c(x^{\gamma \alpha },y^{\gamma \beta })_\lambda = c(x^{\gamma_1 \alpha },y^{\gamma_1\beta })_\lambda c(x^{\gamma_2\alpha },y^{\gamma_2\beta })_\lambda.
$$
From where, if $\gamma = r_0 + r_1 t + \ldots + r_s t^s \in R$, then 
\begin{equation}\label{eq:43}
    c(x^{\gamma \alpha}, y^{\gamma \beta})_\lambda = c(x^{r_0 \alpha}, y^{r_0 \beta})_\lambda c(x^{r_1 t \alpha}, y^{r_1 t \beta})_\lambda \cdots c(x^{r_s t^s \alpha}, y^{r_s t^s \beta})_\lambda.
\end{equation}
Since $r_i \in \mathbb{Q}$, then
\begin{equation}\label{eq:46}
    c(x^{r_i t^i \alpha}, y^{r_i t^i \beta})_\lambda =  c(x^{t^i \alpha}, y^{t^i \beta})_\lambda^{r_i}.
    \end{equation}
In particular,
\begin{equation}\label{eq:46new}
    c(x^{r_i t^i \alpha}, y^{r_i t^i \beta})_t =  c(x^{t^i \alpha}, y^{t^i \beta})_t^{r_i}.
    \end{equation}
Therefore, the set of $R$-generators of $D$ from (\ref{eq:gensD}) can be further reduced using relations (\ref{eq:43}) and (\ref{eq:46new}). To this end, let $P$ denote the set of special representatives of $(\alpha, \beta)$ in the equivalence classes $(\overline{\alpha,\beta})$ from $R_0^2$, and let $B$ denote the following set of c-commutators:
\begin{equation} \label{eq:gensDnew}
B = \{c(x^{\alpha t^k},y^{\beta t^k})_t \mid (\alpha, \beta) \in P, k \in \mathbb{N}\}.
\end{equation}
Then $D = \langle B\rangle_R.$ We claim that $D$ is a free $R$-module with basis $B$. Indeed, according to our construction, we start with a free $R$-module $D$ with basis $\{c(x^\alpha,y^\beta)_\lambda \mid \alpha, \beta, \lambda \in R \}$, and define the $R$-exponentiation in the group $G = N\otimes R = (N\otimes_{\mathcal{H}} R) \times D$ by (\ref{eq:main-def}).
By imposing relations according to conditions E0)-E10), E10'), as well as conditions (\ref{eq:43}) and (\ref{eq:46new}), we guarantee that our $R$-exponentiation satisfies all Axioms 1)-4) of $R$-groups. Moreover, any 2-nilpotent $R$-group $R$-generated by the group $N$ satisfies the same relations. Thus, we have constructed a group $N \otimes R$. Note that all the relations we imposed are of the form when one c-commutator is equal to a product of other c-commutators or their $R$-powers. This is equivalent to removing all c-commutators that appear on the left-hand sides of the relations from the set of generators of the original free $R$-module $D$. If we do this in the proper order, we will end up with a free $R$-module generated by the set $B$. 

\medskip
\noindent
Case 2.  $R= \mathbb{Q}(t)$.

We pass from the ring $\mathbb{Q}[t]$ to the field $\mathbb{Q}(t)$ using the same constructions as above, replacing the ring $\mathbb{Q}[t]$ with the field $\mathbb{Q}(t)$. The only difference in this case is that every nonzero element of the ring $\mathbb{Q}[t]$ is now invertible. Thus, many of the relations from E12) appear.

Following the strategy above, we construct a basis for the free module $D$ in the situation under consideration. Everything proceeds as in the case of the ring $R$, with the relations from E10') rewritten appropriately. Namely, the pairs of the form $P=\{(1,\beta)\mid\;0\neq \beta\in R\}$ can be chosen as generating elements for the equivalence classes of pairs $(\overline{\alpha,\beta})$ in this case, and the additive $\mathbb{Q}$-basis $T=\{t^k\mid\;k\in\mathbb{N}, k \neq 1\}$ (without the unit 1) of the ring $\mathbb{Q}[t]$ is completed to an additive basis $S$ of the field $\mathbb{Q}(t)$ by adding all the simple fractions of the form $\frac{t^k}{p(t)^m}$, where $p(t)$ is an irreducible monic (leading coefficient equal to 1) polynomial over $\mathbb{Q}$, $m$ is a positive integer $0\leq m < deg(p(t))$.
Let
$$
P = \{(1,\beta)s \mid 0\neq \beta \in R, s \in S \}.
$$
Then
\begin{equation} \label{eq:baza-B-2t}
     B=\big\{c(x^\alpha,y^\beta)_t \mid\;(\alpha,\beta)\in P\big\}  \end{equation} 
    is an $R$-basis of the module $D$ in the case $R = \mathbb{Q}(t)$.

Thus we have proved the following result.
\begin{theorem}\label{th:3.1}
A free $2$-step nilpotent $R$-group of rank 2 is $R$-isomorphic to a tensor completion $N \otimes R$ that satisfies the following conditions:
\begin{itemize}
    \item [1)] $N \otimes R$, as an abstract group (not an $R$-group!), is isomorphic to the group $(N \otimes_{\mathcal{H}} R) \times D$, where $N \otimes_{\mathcal{H}} R$ is the Hall $R$-completion of $N$, considered as an abstract group (not an $R$-group), and $D$ is the $R$-module generated by all $c$-commutators $c(x^\alpha,y^\beta)_\mu, \alpha, \beta, \mu \in R$, with defining relations E0)-E14);
    \item [2)] $R$-exponentiation  in the group $(N \otimes_{\mathcal{H}} R) \times D$ is defined by the formula (\ref{eq:exponentiation-new});
    \item [3)] if $R = \mathbb{Q}[t]$, then $D$ is a free $R$-module with basis $B$ introduced in (\ref{eq:gensDnew});
    \item [4)] if $R = \mathbb{Q}(t)$, then $D$ is a free $R$-module with basis $B$ introduced in (\ref{eq:baza-B-2t}).
\end{itemize} 
\end{theorem}

\section{Open problems}

\begin{problem}
Let $R$ be a field of characteristic 0, and $G$ be a finitely generated free 2-nilpotent group. By Theorem 5.1, $D$ is a vector space over $R$. Describe the basis of the vector space $D$ over $R$.
\end{problem}

\begin{problem}
Let $G$ be a free 2-nilpotent group of finite rank, and $R$ be a constructive (recursive) ring, e.g., $R = \mathbb{Q}[t]$, or $R = \mathbb{Q}(t)$, or $R$ is a finite extension of $\mathbb{Q}$. The problem is to develop a rewriting process that, given an element $g \in G \otimes R$ as an $R$-word $w(g_1, \ldots,g_n)$, transforms $w$ into a product $g = \bar u^{\bar\alpha}\bar v^{\bar \beta} d$, where $d \in D$ is defined as a product of c-commutators (with exponents in $R$) in the corresponding free basis of the module $D$, which exists by Theorem \ref{th:5.1}.
\end{problem}

\begin{problem}
Describe the relations between c-commutators in 3-nilpotent $R$-groups, similar to those in Section 3.
\end{problem}

\begin{problem}
Describe the free nilpotent $R$-groups of finite rank in the class $\N_{3,R}$.
\end{problem}

\textbf{Acknowledgments}. This work was supported by the Shota Rustaveli National Science Foundation of Georgia under project FR-21-4713.

%%%%%%%%%%%%%%%%%%%%%%%

\end{document}